\documentclass[10pt]{siamltex}

\usepackage{listings}
\usepackage{amsfonts}
\usepackage{amsmath}
\usepackage{amssymb}
\usepackage{graphicx}
\usepackage{subcaption}
\usepackage{float}
\usepackage{mathtools}
 
\numberwithin{equation}{section}
\numberwithin{figure}{section}

\usepackage{epstopdf}
\usepackage{epsfig}

\usepackage[colorlinks=true, bookmarksopen,
            pdfauthor={CST},
            pdfcreator={pdftex},
            pdfsubject={algorithms},
            linkcolor={blue},
            anchorcolor={black},
            citecolor={blue},
            filecolor={magenta},
            menucolor={black},
            pagecolor={blue},
            plainpages=false,pdfpagelabels,
            urlcolor={blue} ]{hyperref}

\newcommand{\norm}[2][]{||#2||_{#1}}

\newcommand*{\defeq}{\mathrel{\vcenter{\baselineskip0.5ex \lineskiplimit0pt
                     \hbox{\scriptsize.}\hbox{\scriptsize.}}}%
                     =}

\DeclareMathOperator*{\argmax}{arg\,max}

\newtheorem{algorithm}{Algorithm}[section]

\newtheorem{remark}{Remark}[section]

\title{A new embedded variable stepsize, variable order family of low computational complexity}%
\author{Victor DeCaria \and Ahmet Guzel \and William Layton \and Yi Li}
\allowdisplaybreaks

\begin{document}

\maketitle
\begin{abstract}
Variable Stepsize Variable Order (VSVO) methods are the methods of choice to efficiently solve a wide range of ODEs with minimal work and assured accuracy. However, VSVO methods have limited impact in timestepping methods in complex applications due to their computational complexity and the difficulty to implement them in legacy code. We introduce a family of implicit, embedded, VSVO methods that require only one BDF solve at each time step followed by adding linear combinations of the solution at previous time levels. In particular, we construct implicit and linearly implicit VSVO methods of orders two, three and four with the same computational complexity as variable stepsize BDF3. The choice of changing the order of the method is simple and does not require additional solves of linear or nonlinear systems.
\end{abstract}
\section{Introduction}
\emph{This is an expanded and informal version of a manuscript with a similar title submitted for publication.} Variable step, variable order (VSVO) methods, such as Gear's method, have become the methods of choice for numerical simulation of systems of ordinary differential equations, yet have little penetration in applications such as fluid dynamics despite the growing interest in variable stepsize integrators, e.g., \cite{JOHN2010514},\cite{kay2010}, \cite{besier2012},\cite{FAILER2018448}. This may be due to storage limitations and the complexity of solving a nonlinear system to compute each of the order's approximations at each timestep. It may also be due to the cognitive complexity of implementing a VSVO method in an already intricate CFD code.

We present, analyze, and test herein a new embedded family of VSVO methods for solving initial value problems addressing the difficulties in complex applications by post-processing solution data with time filters. Time filters are an established tool to non-intrusively modify weather models to suppress spurious, nonphysical modes to improve predictions \cite{asselin72} with recent improvements \cite{Williams2009},\cite{williams2011},\cite{williams2013}, \cite{litrenchea},\cite{aluthge2016}, \cite{hill2017} and have been used to improve the physical fidelity of artificial compression methods \cite{DECARIA2017733}. They are simple to implement and inexpensive to compute. 

This research is motivated by the algorithm of \cite{guzel} which is (for constant stepsize)
\begin{flalign*}
&\text{Backward Euler}  &&\frac{y^{1}_{n+1}-y_{n}}{k} = f(t_{n+1},y_{n+1}^{1})\\
&\text{Time Filter}  &&y^{2}_{n+1} = y^{1}_{n+1} - \frac{1}{3}(y^{1}_{n+1}-2 y_n +y_{n-1})
\end{flalign*}
The approximation $y^{2}_{n+1}$ is second order, $A$-stable and is implemented by adding one line to an existing Backward Euler code. The first contribution herein is to answer the natural question: Can this be extended by more filters, using more $y$ values, to produce an embedded family of methods and from that a VSVO algorithm of negligible additional complexity over Backward Euler?

We prove in Theorem \ref{thm:nonexistence} an order barrier: Backward Euler can only be made up to second order with linear time filters. Thus, to develop an embedded family of higher accuracy, we apply the time filter idea beginning with a third order method, BDF3. BDF methods are popular in CFD, e.g. BDF2 \cite{fiordilino2017}\cite{akbas} \cite{wang2011}, BDF3 \cite{CHARNYI2017289}\cite{guermond2015}, BDF4 \cite{klein2015}, convex combinations of BDF methods \cite{vatsa2010} \cite{ravindran2015} \cite{jiang2017}, a predictor corrector scheme using BDF \cite{NIGRO2014136}, VSVOBDF \cite{HAY2015151}, and many others. We develop time filters for every variable stepsize, variable coefficient BDF$p$ method that increases their order of accuracy by one. We call this method Filtered BDF$p$+1 (FBDF$p$+1), and it is given in Algorithm \ref{alg:bdf-filter}. We then use the filtering idea to generate a VSVO algorithm of orders 2,3,4 with complexity comparable to BDF3. Starting with the approximation from BDF3, we apply a time filter to obtain FBDF4. We then develop a second filter, BDF3-Stab \eqref{eqn:BDF3-2-variable} to the BDF3 approximation to yield a second order $A$-stable  (and $G$-stable) approximation (see Theorem \ref{thm:G-stab}). 
 Let the super script denote the order of the approximation, and the subscript denote the timestep. The resulting method, Multiple Order One Solve Embedded 234 (MOOSE234) for constant stepsize is

\begin{flalign*}
&\text{BDF3}  &&\frac{11y^{3}_{n+1}-18y_{n} +9y_{n-1} - 2y_{n-2}}{6k} = f(t_{n+1},y_{n+1}^{3})\\
&\text{BDF3-Stab}  &&y^{2}_{n+1} = y^{3}_{n+1} + \frac{9}{125}(y^{3}_{n+1}-3 y_n + 3 y_{n-1} - y_{n-2}) \\ 
&\text{FBDF4}  &&y^{4}_{n+1} = y^{3}_{n+1} - \frac{3}{25}(y^{3}_{n+1}-4 y_n + 6 y_{n-1} - 4y_{n-2} + y_{n-3})\\ \\
\hline
& \text{Error} && Est_2=y^{2}_{n+1}-y^{3}_{n+1}\\
& \text{Estimation}&& Est_3=y^{4}_{n+1}-y^{3}_{n+1}
\end{flalign*}
\begin{gather*}
Est_4 = y^{4}_{n+1} - \frac{48}{25}y_n + \frac{36}{25}y_{n-1}-\frac{16}{25}y_{n-2} + \frac{3}{25}y_{n-3} - \frac{12}{25}kf(t_{n+1},y^{4}_{n+1})
\end{gather*}

The algorithm for variable timestep is given in Section \ref{sec:algorithm}. This is an embedded family; its complexity is dominated by the nonlinear BDF3 solve. The remaining steps contribute negligible cost, require non-intrusive modifications to an existing BDF3 code, and, when used for a complex application, are single instruction, multiple data (SIMD) type, which adapt well to parallel architectures. Each step computes a solution of different temporal orders of accuracy, so $Est_2$ gives an error estimator for $y^2_{n+1}$, and $Est_3$ gives an error estimator for $y^3_{n+1}$. $Est_4$ is an approximation of the leading term of the local truncation error of $y^4_{n+4}$. Thus, the first two error estimates are embedded, and contribute negligible computational expense. Using standard strategies for timestep and order selection, the family yields a VSVO algorithm, presented in Section \ref{sec:algorithm}. The resulting algorithm, MOOSE234, performed well on the Van der Pol test problem in Section \ref{sec:van}. A linearly implicit version of MOOSE234 performed well on the incompressible Navier-Stokes equations, Section \ref{sec:numerical_formulation}.
\section{Preliminaries\label{sec:preliminaries}}
The methods discussed herein will be shown to correspond to one-leg methods (OLMs), which require one function evaluation per timestep. OLMs are distinct from linear multistep methods (LMMs), have better nonlinear stability properties than their LMM counterparts and are well suited to time adaptation \cite{dahlquist79}, \cite{Nevanlinna1978}, \cite{janssen}, \cite{Kulikov2006}. Section \ref{subsec:notation} introduces the basic notation needed to state the methods in their full variable stepsize generality.  Section \ref{subsec:classical_error} recalls a result of Dahlquist for OLMs needed herein.
\subsection{Notation\label{subsec:notation}}
Let $k_n = t_{n+1}-t_n$, and $\bar{k}_n$ be homogeneous of first degree in $k_n,k_{n+1},...,k_{n+m-1}$. An $m$ step OLM is \cite{dahlquist79}
\begin{equation}\label{eqn:olm}
\sum_{i=0}^{m} \alpha_i y_{n+i} = \bar{k}_n f\left(\sum_{i=0}^m \beta_i t_{n+i},\sum_{i=0}^m \beta_i y_{n+i}\right).
\end{equation}
The corresponding LMM twin is 
\begin{equation}\label{eqn:lmm}
\sum_{i=0}^{m} \alpha_i y_{n+i} = \bar{k}_n \sum_{i=0}^m\beta_i f\left(t_{n+i},y_{n+i}\right).
\end{equation}
The first and second characteristic polynomials associated with both OLMs and LMMs are
\begin{equation}\notag
\rho(r) = \sum_{j=0}^{m} \alpha_j r^j \hspace{10mm} \sigma(r) = \sum_{j=0}^{m} \beta_j r^j.
\end{equation}
For variable timesteps, $\alpha_i$ and $\beta_i$ depend on the stepsizes $k_n$, but we suppress additional subscripts for the sake of notation. In the construction of the methods developed herein, it is convenient to absorb $\bar{k}_n$ into $\alpha_i$ through the change of variables $\bar{\alpha}_i = \alpha_i/\bar{k}_n$. Then \eqref{eqn:olm} becomes
\begin{equation}\label{eqn:olm_bar}
\sum_{i=0}^{m} \bar{ \alpha}_i y_{n+i} =  f\left(\sum_{i=0}^m \beta_i t_{n+i},\sum_{i=0}^m \beta_i y_{n+i}\right).
\end{equation}

The variable stepsize, variable coefficient BDF methods of order $p$ (BDF$p$) using Newton interpolation are given as follows. 
Using the notation of \cite[pg. 155]{HAY2015151}, let $\delta^j$ be the $j$th order backward divided difference defined by
\begin{equation}\notag
\delta^j\phi = \phi[t_{n+m},t_{n+m-1},\cdots,t_{n+m-j}].
\end{equation}
Let $m\geq p$. Then BDF$p$ can be written
\begin{equation}\notag
\sum_{j=m-p}^m \bar{\alpha}^{(p)}_j y_{n+j} \defeq \sum_{j=1}^p \left[\prod_{i=1}^{j-1}(t_{n+m} - t_{n+m-i}) \right]\delta^j y = f(t_{n+m},y_{n+m})
\end{equation}
where the $\bar{\alpha}^{(p)}_js$ are given implicitly by the equation. The coefficients for variable stepsize BDF up to order four have also been written explicitly in terms of stepsize ratios in \cite{Wang2008}. The divided differences can be expanded as 
\begin{equation}\label{eqn:divided-diff-expansion}
\delta^j \phi = \sum_{i=m-j}^m c^{(j)}_i \phi_{n+i}
\end{equation}
with a procedure to generate the $c^{(j)}_i$s given in the appendix.

If the approximation at $t_{n+m}$ is an intermediate approximation, we will denote it by $\tilde{y}_{n+m}$ to indicate a generic intermediate approximation, or $y^{p}_{n+m}$ to indicate an approximation of order $p$. The intermediate divided differences are defined
\begin{equation}\notag
\delta^j \tilde{y} = c^{(j)}_m\tilde{y}_{n+m} +  \sum_{i=m-j}^{m-1} c^{(j)}_i y_{n+i},
\end{equation}
\begin{equation}\notag
\delta^j y^{p} = c^{(j)}_my^{p}_{n+m} +  \sum_{i=m-j}^{m-1} c^{(j)}_i y_{n+i},
\end{equation}
\textbf{Example.} Recalling the variable stepsize Backward Euler plus time filter in \cite[pg. 307]{guzel}, we will rewrite it in terms of divided differences. With $\tau=\frac{k_{n+1}}{k_n}$ it is as follows, 
\begin{flalign*}
&\text{Backward Euler}  &&\frac{y^{1}_{n+2}-y_{n+1}}{k_{n+1}} = f(t_{n+2},y_{n+2}^{1})\\
&\text{Time Filter}  &&y_{n+2} = y^{1}_{n+2} - \frac{\tau(1+\tau)}{1+2\tau}\left(\frac{1}{1+\tau}y^{1}_{n+2}-y_{n+1}+ \frac{\tau}{1+\tau} y_{n}\right)
\end{flalign*}
Through algebraic manipulation, this can be written with divided differences,
\begin{flalign}
&\text{Backward Euler}  &&\delta^1y^{1} = f(t_{n+2},y_{n+2}^{1})\notag\\
&\text{Time Filter}  &&y_{n+2} = y^{1}_{n+2} - \left(\frac{k_{n+1}}{\frac{1}{k_{n+1}}+\frac{1}{k_{n+1}+k_n}}\right)\delta^2 y^{1} \label{eqn:example_div_diff} \\
& && = y^{1}_{n+2} - \left(\frac{k_{n+1}}{\frac{1}{k_{n+1}}+\frac{1}{k_{n+1}+k_n}}\right)\left(\frac{\frac{y^{1}_{n+2}-y_{n+1}}{k_{n+1}} - \frac{y_{n+1}-y_{n}}{k_{n}}}{k_{n+1} + k_n}\right) \notag
\end{flalign}
The divided differences are expanded explicitly as
\begin{equation}\notag
\delta^1y^{1} = c^{(1)}_2y_{n+2}^{1} + c^{(1)}_1y_{n+1} +  c^{(1)}_0y_{n} = \frac{1}{k_{n+1}}y_{n+2}^{1} + \frac{-1}{k_{n+1}}y_{n+1} +  0y_{n}
\end{equation}
\begin{gather*}\notag
\delta^2y^{1} = c^{(2)}_2y_{n+2}^{1} + c^{(2)}_1y_{n+1} +  c^{(2)}_0y_{n} = \\=
\frac{1}{k_{n+1}(k_{n+1} + k_n)}y_{n+2}^{1} + \left(\frac{-1}{k_{n+1}} + \frac{-1}{k_n}\right)\frac{1}{k_{n+1}+k_n}y_{n+1} +  \frac{1}{k_{n}(k_{n+1} + k_n)}y_{n}.
\end{gather*}
The coefficients $c^{(j)}_i$ are apparent. The rearrangement \eqref{eqn:example_div_diff} is a special case of Algorithm \ref{alg:bdf-filter} with $p=1$. 

The coefficients for higher order differences are lengthy to write out. While they may be hard-coded into a program, they are easily computed with recursion \cite[pg. 175]{HAY2015151} or nested loops (see Appendix \ref{sec:psuedocode}). 
\subsection{Classical Error Results\label{subsec:classical_error}}
Define the local truncation error
\begin{align*}
\varepsilon_n = \bar{k}_n^{-1}\sum_{i=0}^{m} \alpha_i y(t_{n+i}) - f\left(\sum_{i=0}^m \beta_iy(t_{n+i})\right)\\
=\sum_{i=0}^{m} \bar{\alpha}_i y(t_{n+i}) - f\left(\sum_{i=0}^m \beta_iy(t_{n+i})\right)
\end{align*} 
A characterization of the accuracy of the OLM can be stated in terms of the differentiation and interpolation error operators \cite{dahlquist81} \cite{dahlquist83}.
\begin{definition}[Differentiation and Interpolation Error \label{def:diff-int-error}]

The differentiation error $L_d$ and interpolation error $L_i$ operators are defined
\begin{align}
&(L_d\phi)\left(\sum_{i=0}^m \beta_i t_{n+i}\right) = \sum_{i=0}^m \bar{\alpha}_i \phi(t_{n+i}) - \phi' \left(\sum_{i=0}^m \beta_it_{n+i}\right)\label{eqn:diff-error}\\
&(L_i\phi)\left(\sum_{i=0}^m \beta_i t_{n+i}\right) = \sum_{i=0}^m \beta_i \phi(t_{n+i}) - \phi \left(\sum_{i=0}^m \beta_it_{n+i}\right).\label{eqn:int-error}
\end{align}
\end{definition}
A Taylor series calculation gives
\begin{equation}\label{eqn:lte-est}
\text{Leading term of }\varepsilon_n = \left( L_dy - f_y\left(\sum_{i=0}^m \beta_i t_{n+i},y\left(\sum_{i=0}^m \beta_i t_{n+i}\right)\right)\cdot L_iy\right)\left(\sum_{i=0}^m \beta_i t_{n+i}\right)
\end{equation}

\begin{definition}[Proposed in {\cite[pg. 8]{dahlquist81}}]\label{def:annihilating}

Let $p_d$ be the largest integer such \ref{eqn:diff-error} is zero for all polynomials of degree $p_d$, and $p_i$ the largest integer such that \ref{eqn:int-error} is zero for all polynomials of degree $p_i$. The order of the one leg method is $\min (p_d,p_i+1)$.
\end{definition}
\section{Embedding BDF$p$ in a new family\label{sec:derivation}}

We now develop a time filter increases the order of consistency of BDF$p$ by one. Let $m = p+1$, and consider the following method.
\begin{algorithm}[Filtered BDF$p$+1 (FBDF$p$+1)]\label{alg:bdf-filter}
Given $$\{y_n,y_{n+1},...,y_{n+m-1}\},$$ find $y_{n+m}$ satisfying
\begin{equation}\label{eqn:bdf-temp}
\sum_{j=1}^p \left[\prod_{i=1}^{j-1}(t_{n+m} - t_{n+m-i}) \right]\delta^j y^p = f(t_{n+m},y^p_{n+m}).
\end{equation}
\begin{equation}\label{eqn:eta}
\eta^{(p+1)} = \frac{\prod_{i=1}^p(t_{n+m}-t_{n+m-i})}{\sum_{j=1}^{p+1}(t_{n+m}-t_{n+m-j})^{-1}}.
\end{equation}\begin{equation}\label{eqn:bdf-filter}
y_{n+m} = y^p_{n+m} - \eta^{(p+1)} \delta^{p+1}y^p.
\end{equation}
\end{algorithm} 
It will be shown in Theorem \ref{thm:fbdf-order} that this method has consistency error of order $p+1$. The proof requires reducing the above steps to an OLM. Specifically, the OLM approximates $y'$ the same as BDF$p+1$.
\begin{algorithm}\label{alg:bdf-filter-equivalent-olm}
Given $$\{y_n,y_{n+1},...,y_{n+m-1}\},$$ and $\eta^{(p+1)}$ as in \eqref{eqn:eta}, find $y_{n+m}$ satisfying
\begin{equation}\label{eqn:bdf-filter-equivalent-olm}
\sum_{j=1}^{p+1} \left[\prod_{i=1}^{j-1}(t_{n+m} - t_{n+m-i}) \right]\delta^j y = f\left(t_{n+m},y_{n+m} + \frac{\eta^{(p+1)}}{1-\eta^{(p+1)} c^{(p+1)}_m} \delta^{p+1}y\right).
\end{equation}
\end{algorithm} 
The $\beta_i$s for this method are given implicitly by $$\sum_{j=0}^m \beta_i y_{n+j} = y_{n+m} + \frac{\eta^{(p+1)}}{1-\eta^{(p+1)} c^{(p+1)}_m} \delta^{p+1}y.$$ Since $\delta^pt=0$ for $p\geq 2$, we see that it is indeed an OLM of the form \eqref{eqn:olm_bar}.

\begin{proposition}
Algorithms \ref{alg:bdf-filter} and \ref{alg:bdf-filter-equivalent-olm} are equivalent.
\end{proposition}
\begin{proof}
Subtract $\eta^{(p+1)} c^{p+1}_m y_{n+m}$ from both sides of \eqref{eqn:bdf-filter}.
\begin{equation}\notag
(1-\eta^{(p+1)} c^{(p+1)}_m)y_{n+m} = (1-\eta^{(p+1)} c^{(p+1)}_m)y^p_{n+m} - \eta^{(p+1)} \delta^{p+1}y.
\end{equation}
Solving for $y^p_{n+m}$ gives
\begin{equation}\label{eqn:solved_for_temp_value}
y^p_{n+m} = y_{n+m} + \frac{\eta^{(p+1)}}{1-\eta^{(p+1)} c^{(p+1)}_m} \delta^{p+1}y.
\end{equation}
Substituting this into the BDF$p$ step \eqref{eqn:bdf-temp}, the right hand side of \eqref{eqn:bdf-temp} becomes the right hand side of \eqref{eqn:bdf-filter-equivalent-olm} as desired. The left hand side of \eqref{eqn:bdf-temp} becomes
\begin{gather}\label{eqn:substituted_lhs}
\sum_{j=1}^p \left[\prod_{i=1}^{j-1}(t_{n+m} - t_{n+m-i}) \right]\delta^j y \\+ \left\lbrace \frac{\eta^{(p+1)}}{1-c_m^{(p+1)}\eta^{(p+1)}}\sum_{j=1}^p \left[\prod_{i=1}^{j-1}(t_{n+m} - t_{n+m-i}) \right]c_m^{(j)}\right\rbrace\delta^{p+1}y. \notag
\end{gather}
We next simplify the scalar, shown in braces, multiplying $\delta ^{p+1}y$ in \eqref{eqn:substituted_lhs}. First, note that a simple calculation (not shown) gives
\begin{equation}\notag
c^{(j)}_m = \left(\prod_{i=1}^{j}(t_{n+m}-t_{n+m-i})\right)^{-1}.
\end{equation}
Splitting the term in braces apart, we see that
\begin{gather*}
\frac{\eta^{(p+1)}}{1-c_m^{(p+1)}\eta^{(p+1)}}=\frac{ \frac{\prod_{i=1}^p(t_{n+m}-t_{n+m-i})}{\sum_{j=1}^{p+1}(t_{n+m}-t_{n+m-j})^{-1}}}{1- \left(\prod_{i=1}^{p+1}(t_{n+m}-t_{n+m-i})\right)^{-1} \frac{\prod_{i=1}^p(t_{n+m}-t_{n+m-i})}{\sum_{j=1}^{p+1}(t_{n+m}-t_{n+m-j})^{-1}}}=\\
=\frac{\prod_{i=1}^p(t_{n+m}-t_{n+m-i})}{\sum_{j=1}^{p+1}(t_{n+m}-t_{n+m-j})^{-1}- (t_{n+m}-t_n)^{-1}}=\frac{\prod_{i=1}^p(t_{n+m}-t_{n+m-i})}{\sum_{j=1}^{p}(t_{n+m}-t_{n+m-j})^{-1}}  
\end{gather*}
and 
\begin{gather*}
\sum_{j=1}^p \left[\prod_{i=1}^{j-1}(t_{n+m} - t_{n+m-i}) \right]c_m^{(j)}\\= \sum_{j=1}^p \left[\prod_{i=1}^{j-1}(t_{n+m} - t_{n+m-i}) \right]\left(\prod_{i=1}^{j}(t_{n+m}-t_{n+m-i})\right)^{-1} = \sum_{j=1}^{p}(t_{n+m}-t_{n+m-j})^{-1}.
\end{gather*}
Thus, the term in braces simplifies to
\begin{equation}
\sum_{j=1}^p \left[\prod_{i=1}^{j-1}(t_{n+m} - t_{n+m-i}) \right]c_m^{(j)}\frac{\eta^{(p+1)}}{1-c_m^{(p+1)}\eta^{(p+1)}} = \prod_{i=1}^{p}(t_{n+m} - t_{n+m-i}).
\end{equation}
Absorbing this into the sum in \eqref{eqn:substituted_lhs} gives the desired left hand side of \eqref{eqn:bdf-filter-equivalent-olm}.
\end{proof}

\subsection{Stability and Error Analysis of FBDF$p$+1\label{sec:error_analysis}}

By construction, the left hand side of FBDF$p$+1 is that of BDF$p$+1, and the argument of the function evaluation in FBDF$p$+1 is a consistent approximation $y(t^{n+m})$. This observation makes it easy to deduce the 0-stability and consistency error of the methods. It is immediate that FBDF$p$+1 is 0-stable whenever BDF$p$+1 is 0-stable since they have the same first characteristic polynomial. Therefore, FBDF6, which is BDF5 plus a time filter, is the highest order 0-stable method for constant stepsize. The 0-stability of variable step methods is highly nontrivial, and conditions on the stepsize ratios to guarantee 0-stability for BDF methods are given \cite{Calvo1990}, and \cite{Guglielmi2001} which improved the upper bound on the stepsize ratios for BDF3. 
The consistency error analysis, presented next, is an application of Definition \ref{def:annihilating}. 
\begin{theorem}\label{thm:fbdf-order} FBDF$p$+1 is consistent of order $p+1$.
\end{theorem}
\begin{proof}
The differentiation error $L_d$ (see Definition \ref{def:diff-int-error}) of FBDF$p$+1 is the same as the local truncation error of BDF$p$+1 which annihilates polynomials up to order $p+1$, so $p_d = p+1$. The interpolation error $L_i$ is
\begin{multline}\label{eqn:interp-error-fbdf}
(L_i\phi)\left(\sum_{i=0}^m \beta_i t_{n+i}\right) = \sum_{i=0}^m \beta_i \phi(t_{n+i}) - \phi \left(\sum_{i=0}^m \beta_it_{n+i}\right)\\
=\phi(t_{n+m}) + \frac{\eta^{(p+1)}}{1-\eta^{(p+1)} c^{(p+1)}_m} \delta^{p+1}\phi - \phi(t_{n+m}) = \frac{\eta^{(p+1)}}{1-\eta^{(p+1)} c^{(p+1)}_m} \delta^{p+1}\phi.
\end{multline}
If $\phi$ is smooth, then for some $\xi \in (t_{n},t_{n+m})$,
\begin{equation}\notag
\delta^{p+1}\phi = \phi[t_{n+m},t_{n+m-1},\cdots,t_n] = \frac{\phi^{(p+1)}(\xi)}{(p+1)!},
\end{equation}
see e.g. \cite{isaacson}. Hence, $(L_i\phi)\left(\sum_{i=0}^m \beta_i t_{n+i}\right)$ is zero on polynomials of degree less than or equal to $p$, so $p_i = p$. Thus, the order of the the method is $\min(p_d,p_i+1) = p+1$.
\end{proof}

\subsection{Error Estimation\label{sec:error_est}}
Let $y^{p+1}_{n+m}$ theFBDF$p$+1 approximation. As a consequence FBDF$p$+1 being $\mathcal{O}(k^{p+1})$, we have that $$Est_p = y^{p+1}_{n+m}-y^p_{n+m}$$ (see Algorithm \ref{alg:bdf-filter}) is an estimate for the local error in BDF$p$. In order to estimate the local error for FBDF$p$+1, we need to estimate \eqref{eqn:lte-est} which involves the exact solution. To approximate $L_i y$, note from \eqref{eqn:solved_for_temp_value} and \eqref{eqn:interp-error-fbdf}
\begin{equation}\notag
L_iy(t) = \frac{\eta^{(p+1)}}{1-\eta^{(p+1)} c^{(p+1)}_m} \delta^{p+1}y(t) \approx y^p_{n+m}-y^{p+1}_{n+m} = -Est_p.
\end{equation}
To approximate the rest of \eqref{eqn:lte-est}, use $y(t_{n+i}) \approx y_{n+i}$ a possible error estimate of FBDF$p$+1 is
\begin{gather}\label{eqn:error_est_fbdf}
Est_{p+1} = \bigg(\sum_{j=1}^{p+1} \left[\prod_{i=1}^{j-1}(t_{n+m} - t_{n+m-i}) \right]\delta^j y^{p+1} - f(t_{n+m},y^{p+1}_{n+m})\\
  +  f_y(t_{n+m},y^{p+1}_{n+m})\cdot Est_p\bigg)\bigg/\bar{\alpha}^{(p+1)}_m \notag
\end{gather}

However, a more successful estimator in our tests was to only estimate $L_d$ with
\begin{gather}\label{eqn:error_est_fbdf_difference}
Est_{p+1} = \bigg(\sum_{j=1}^{p+1} \left[\prod_{i=1}^{j-1}(t_{n+m} - t_{n+m-i}) \right]\delta^j y^{p+1} - f(t_{n+m},y^{p+1}_{n+m})\bigg)\bigg/\bar{\alpha}^{(p+1)}_m .
\end{gather}
\eqref{eqn:error_est_fbdf} seemingly underpredicted the error, which lead to overly large stepsizes. Using \eqref{eqn:error_est_fbdf_difference} is pessimistic because it does not allow for cancellation with the neglected term. It's success may be due to the fact that enforcing small $Est_{p+1}$ also forces the $y^{p+1}_{n+m}$ to make the BDF$p$+1 residual small. Thus, the solution, up to a scaling, is within $\varepsilon$ of satisfying a nearby method of order $p$
\subsection{Order Barrier\label{sec:neg_result}}
We prove in this section that it is impossible to increase the order of the simplest method, BDF1, by more than one by filtering $y_{n+m}$ alone. It is an open problem whether this is generally true for other methods.

\begin{theorem}[An Order Barrier]\label{thm:nonexistence}
BDF1 followed by a linear time filter of arbitrary finite length, 
\begin{equation}\notag
\tilde{y}_{n+m}- y_{n+m-1} = k f(\tilde{y}_{n+m})
\end{equation}
\begin{equation}\label{eqn:arbitrary_be_filter}
y_{n+m} = \tilde{y}_{n+m} - (c_m\tilde{y}_{n+m} + c_{m-1}y_{n+m-1} + \cdots +c_0y_n) \hspace{10mm} c_m\neq 0
\end{equation}
is of no higher than second order consistency.
\end{theorem}
\begin{proof}For simplicity, we assume the ODE is autonomous, and that the step size is constant with $k_n=k$ for all $n$. Without loss of generality, we need only consider the application of one time filter; a sequence of time filters can be reduced to one of possibly greater length by condensing the intermediate values.

Solving \eqref{eqn:arbitrary_be_filter} for $\tilde{y}_{n+m}$,
\begin{equation}\notag
\tilde{y}_{n+m} = \frac{1}{1-c_m}(y_{n+m}+c_{m-1}y_{n+m-1} + \cdots +c_0y_n ).
\end{equation}
For convenience, we perform the change of variables $\hat{c}_m = (1-c_m)^{-1}$, $\hat{c}_{m-1} = (c_{m-1})(1-c_m)^{-1}-1$, and $\hat{c}_i=c_i(1-c_m)^{-1}$ for $0\leq i \leq m-2$.
\begin{equation}\notag
\tilde{y}_{n+m} = y_{n+m-1} + \sum_{j=0}^m \hat{c}_j y_{n+j}
\end{equation}
Substituting this into the backward Euler step yields the equivalent OLM
\begin{equation}\notag
\sum_{j=0}^m \hat{c}_j y_{n+j} = k_{n+m-1}f(y_{n+m-1} + \sum_{j=0}^m \hat{c}_j y_{n+j}).
\end{equation}

$L_d$ must be zero on polynomials of order 3, and $L_i$ must be zero on polynomials of order 2. We derive conditions on $\hat{c}_i$ from the differential error operator applied to $1,t$, and $t^2$.
\begin{equation}\notag
(L_d1)(t_{n+m}) = k^{-1}\sum_{i=0}^m\hat{c}_i = 0 \Longrightarrow \sum_{i=0}^m\hat{c}_i=0
\end{equation}
\begin{equation}\notag
(L_dt)(t_{n+m}) = k^{-1}\sum_{i=0}^m\hat{c}_i(n+ik) - 1 = \sum_{i=0}^m\hat{c}_ii - 1=0 \Longrightarrow \sum_{i=0}^m\hat{c}_ii = 1
\end{equation}
\begin{align*}
(L_dt^2)(t_{n+m}) &= k^{-1}\sum_{i=0}^m\hat{c}_i(n+ik)^2 - 2(n+mk) = k^{-1}\sum_{i=0}^m\hat{c}_i(n^2+2nik + i^2k^2)\\ &- 2(n+mk)
=2n\sum_{i=0}^m\hat{c}_ii + k\sum_{i=0}^m\hat{c}_ii^2 - 2(n+mk)=0 \Rightarrow \sum_{i=0}^m\hat{c}_ii^2 = 2m
\end{align*}
Next, apply the interpolation error operator to $t^2$.
\begin{align*}
(L_it^2)(t_{n+m})= \sum_{i=0}^m\hat{c}_i(n+ik)^2 + (n+(m-1)k)^2 - (n+km)^2\\
=\sum_{i=0}^m\hat{c}_in^2 + \sum_{i=0}^m\hat{c}_i2nik + \sum_{i=0}^m\hat{c}_i i^2k^2 + n^2 + 2n(m-1)k + (m-1)^2k^2\\
-n^2-2kmn-k^2m^2\\
=2nk\sum_{i=0}^m\hat{c}_ii + k^2\sum_{i=0}^m\hat{c}_ii^2 + 2nmk-2nk+(m^2-2m+1)k^2\\
-2kmn-k^2m^2 = 2nk+2mk^2-2nk+m^2k^2 -2mk^2 +k^2 -k^2m^2 = k^2.
\end{align*}
The operator does not vanish on quadratics, so by Definition \ref{def:annihilating}, the method is no higher than second order.

\end{proof}
\begin{remark}
A sequence of linear post-filters like \eqref{eqn:arbitrary_be_filter} can be reduced to one filter of the same type. Thus, this result means that the order of Backward Euler cannot be improved beyond $\mathcal{O}(k^2)$ by adding more filters.
\end{remark}
\subsection{Stabilizing Time Filters\label{sec:bdf3stab}}
Adaptive codes using non-$A$-stable methods will, when necessary, decrease timesteps to enforce stability rather than accuracy. We thus need an $A$-stable member in the embedded family. This is achieved automatically for BDF1 plus a time filter (FBDF2), but we are limited to lower order since filters can increase the order of accuracy by only one (see Theorem \ref{thm:nonexistence}). Thus, we construct time filters that create lower order, but $A$-stable, embedded methods from BDF3 for constant stepsize. The result (BDF3-Stab below) is second order, $G$-stable and therefore $A$-stable by \cite{dahlquist78}. We generalize the filter for variable stepsize. We begin with the ansatz that such a filter should be a third order perturbation of BDF3.
\begin{algorithm}[Constant Stepsize Stabilized BDF3, BDF3-Stab]
\begin{flalign}
&\text{BDF3 Step}  &&\frac{11y^{3}_{n+3}-18y_{n+2} +9y_{n+1} - 2y_{n}}{6k} = f(t_{n+3},y_{n+3}^{3}) \notag \\
&\text{BDF3-Stab Step}  &&y_{n+3} = y^{3}_{n+3} + \mu(y^{3}_{n+3}-3 y_{n+2} + 3 y_{n+1} - y_{n})\label{eqn:BDF3-stab-filter}
\end{flalign}
\end{algorithm}
The induced OLM after eliminating the intermediate variable is
\begin{multline}\label{BDF3-2}
\frac{11y_{n+3}-18y_{n+2} +9y_{n+1} - 2y_{n}}{6k} - \frac{11}{6k}\frac{\mu}{1+\mu} (y_{n+3}-3 y_{n+2} + 3 y_{n+1} - y_{n})\\= f\left(t^{n+3},y_{n+3} - \frac{\mu}{1+\mu} (y_{n+3}-3 y_{n+2} + 3 y_{n+1} - y_{n}) \right)
\end{multline}
From \eqref{BDF3-2}, a Taylor series calculation shows that the method is second order consistent with 
\begin{equation}\notag
\text{leading term of }\varepsilon^n = -\frac{11}{6}\frac{\mu}{1+\mu}y'''k^2.
\end{equation}  What remains is to show that it may be $G$-Stable for a range of $\mu$. 
\begin{theorem}\label{thm:G-stab}
BDF3-Stab is $G$-Stable for $\mu \in$ [0.07143215,0.14285528]
\end{theorem}
\begin{proof}
Multiplying \eqref{BDF3-2} by $\frac{6k}{11}$ gives
\begin{multline}\notag
\frac{y_{n+3}}{1+\mu} + \left(\frac{3\mu}{1+\mu}-\frac{18}{11}\right)y_{n+2} + \left(- \frac{3\mu}{1+\mu}+\frac{9}{11}\right)y_{n+1} +\left(\frac{\mu}{1+\mu}-\frac{2}{11}\right)y_{n} \\
=\frac{6}{11}kf\left(t_{n+3},\frac{y_{n+3}}{1+\mu} + \frac{3\mu}{1+\mu}y_{n+2} - \frac{3\mu}{1+\mu}y_{n+1} + \frac{\mu}{1+\mu}y_{n} \right).
\end{multline}
Define the G-matrix as
\[
\left(\begin{array}{ccc}
g33 &g32 &g31\\
g32 &g22 &g21\\
g31 &g21 &g11
\end{array}\right),
\]
with the associated $G$-norm given by $\|v\|^2_{G}=(v,Gv)$.

We show there exists a symmetric positive definite $G$ and some constants $a_3$, $a_2$, $a_1$, $a_0$ such that 
\begin{eqnarray*}
\begin{aligned}
&\left(\frac{y_{n+3}}{1+\mu} + \left(\frac{3\mu}{1+\mu}-\frac{18}{11}\right)y_{n+2} + \left(- \frac{3\mu}{1+\mu}+\frac{9}{11}\right)y_{n+1} +\left(\frac{\mu}{1+\mu}-\frac{2}{11}\right)y_{n}\right) \\
&\cdot\left(\frac{y_{n+3}}{1+\mu} + \frac{3\mu}{1+\mu}y_{n+2} - \frac{3\mu}{1+\mu}y_{n+1} + \frac{\mu}{1+\mu}y_{n} \right)\\
& =\left\|\begin{array}{c}
y_{n+3}\\
y_{n+2}\\
y_{n+1}
\end{array}\right\|_G^2-\left\|\begin{array}{c}
y_{n+2}\\
y_{n+1}\\
y_{n}
\end{array}\right\|_G^2 +\|a_3 y_{n+3}+a_2 y_{n+2} +a_1 y_{n+1}+a_0 y_{n}\|^2.
\end{aligned}
\end{eqnarray*}
By matching the coefficients of  $y_i$ on the left hand side and right hand side, we get the following equations:
\begin{equation}\label{lemma1.1}
\left\{
             \begin{array}{lr}
             g_{33}+(a_3)^2=(\frac{1}{1+\mu})^2, &  \\
             g_{22}-g_{33}+(a_2)^2=\frac{45\mu^2-54\mu}{11(1+\mu)^2},& \\
             g_{11}-g_{22}+(a_1)^2=\frac{72\mu^2-27\mu}{11(1+\mu)^2}, &  \\
             (a_0)^2-g_{11}=\frac{9\mu^2-2\mu}{11(1+\mu)^2},&\\
             2g_{32}+2a_3a_2=\frac{48\mu-18}{11(1+\mu)^2},&\\
             2g_{31}+2a_3a_1=\frac{-57\mu+9}{11(1+\mu)^2},&\\
             2a_3a_0=\frac{20\mu-2}{11(1+\mu)^2},&\\
             2g_{21}-2g_{32}+2a_2a_1=\frac{-117\mu^2+81\mu}{11(1+\mu)^2},&\\
             -2g_{31}+2a_2a_0=\frac{42\mu^2-24\mu}{11(1+\mu)^2},&\\
             -2g_{21}+2a_1a_0=\frac{-51\mu^2+15\mu}{11(1+\mu)^2}.&\\

             \end{array}
\right.
\end{equation}
We solved for $g_{ij}$ and $a_i$ in terms of $\mu$ with MATLAB, and write them explicitly in the appendix. Using Sylvester’s criterion, we seek an interval of $\mu$ for which the principle minors of $G$ have positive determinant. We denote 
\begin{gather}
\displaystyle G_1=g _{33},\ \ G_2=\left|\begin{array}{cc}
g_{33} &g_{32}\\
g_{23} &g_{22}
\end{array}\right|, \ \ \text{and} \ \ G_3=\det(G).
\end{gather}
\begin{figure}
\centering

\begin{subfigure}{0.49\linewidth}
   \centering
   \includegraphics[width = 1\linewidth]{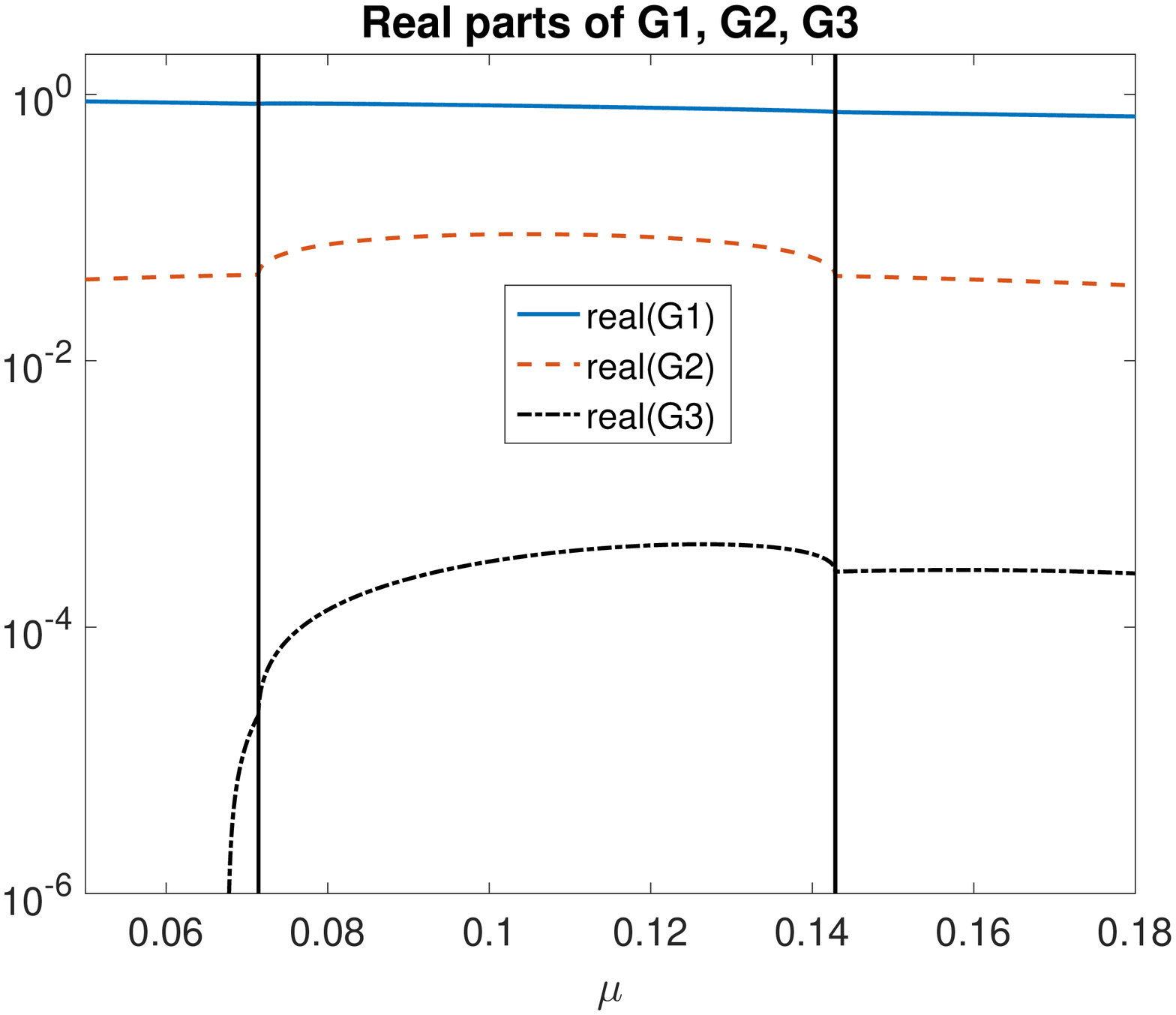}
   \caption{ \label{mu2}}
\end{subfigure}
\begin{subfigure}{0.49\linewidth}
   \centering
   \includegraphics[width = 1\linewidth]{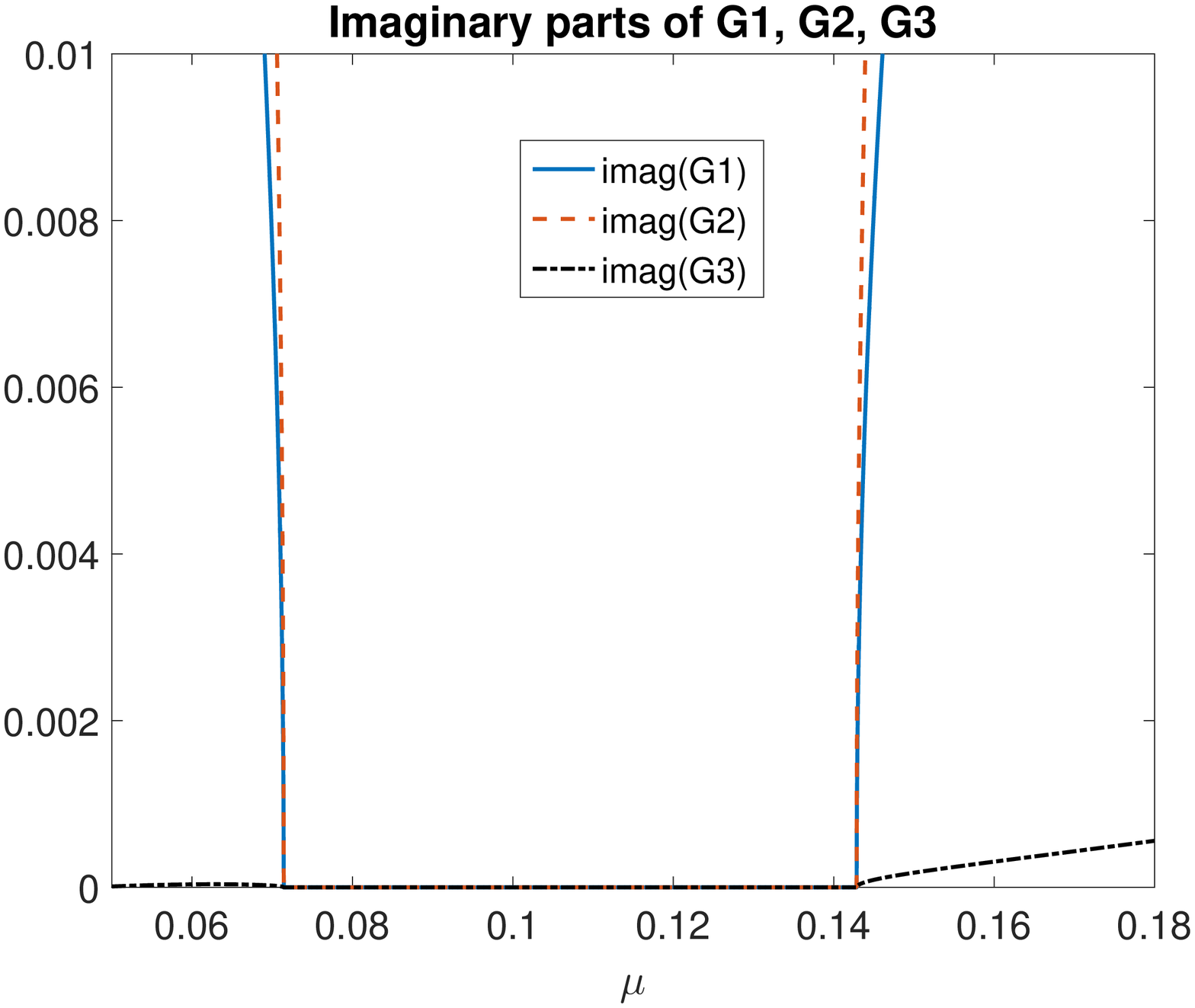}
   \caption{ \label{mu3}}
\end{subfigure}
\caption{The real parts of $G1, G2$ and $G3$ are positive and the imaginary parts vanish in the region of $G$ stability, which is bounded by the vertical bars in Fig. \ref{mu2}. \label{fig:mu}\selectfont{}}
\end{figure} 
By plotting the real and imaginary parts of $G1$, $G2$, and $G3$, we see that the real parts are positive, and the imaginary parts vanish within the interval $\mu \in$ [0.07143215,0.14285528], which implies that BDF3-Stab is $G$ stable for $\mu$ in this interval. 

\end{proof}
Dahlquist's Second Barrier states that the leading error constant $C$ of all $A$-Stable methods is, in magnitude, greater than or equal to $C_{TR}=1/12$, which is attained by the trapezoid rule \cite{dahlquist1963}. For the left endpoint of the interval given in Theorem \ref{thm:G-stab}, BDF3-Stab has a leading error constant of $\frac{1}{12} < C_{\text{BDF3-Stab}}\approx0.1222 < \frac{2}{12}$. This is about 2.73 times smaller than $C_{BDF2}=\frac{4}{12}$ (compare with the optimized blended BDF2/BDF3 and BDF2/BDF3/BDF4 schemes in \cite{vatsa2010} which have respectively 2 and 2.64 times smaller leading constants than $C_{BDF2}$).

The filter is extended to variable stepsize by replacing $y^{3}_{n+3}-3 y_{n+2} + 3 y_{n+1} - y_{n}$ with the third order divided difference $\delta^j$, and rescaling by the leading coefficient $c^{(3)}_3$.
\begin{equation}\label{eqn:BDF3-2-variable}
y^{2}_{n+3} = y^{3}_{n+3} + \frac{\mu}{c^{(3)}_3}\delta^3y^{3},
\end{equation}
 This gives a variable stepsize BDF3 method a stable method to switch to rather than cutting the timestep. Other extensions of this filter to variable stepsize, or $G$-Stabilizing filters using more previous values to reduce the error constant may be possible.
 \begin{remark}
 Note this extension to variable stepsize is not unique. We simply rescaled the third divided difference with something that makes it consistent with \eqref{eqn:BDF3-stab-filter} for constant time step. In this case, we multiply by $\mu/c^{(3)}_3$. Other possible extensions that were not tested, and in no particular order, are
 \begin{equation}\notag
y^{2}_{n+3} = y^{3}_{n+3} + \frac{\mu}{c^{(3)}_0}\delta^3y^{3},
\end{equation}
 \end{remark}
\section{The VSVO Algorithm\label{sec:algorithm}}
In Section \ref{sec:derivation}, we derived a general embedded \\method FBDF$p$+1 which increases the order of any BDF$p$ method by one, and an embedded stabilized BDF3 method (BDF3-Stab) that is second order and $G$-Stable. We combine FBDF4 and BDF3-Stab to create an embedded implicit method that is second, third, and fourth order called Multiple Order One Solve Embedded 234 (MOOSE234). 
After a filter is applied, the new solution is used to estimate the error in the pre-filtered solution. We then pick the solution which allows for the largest stepsize to be taken. The variable stepsize algorithm is as follows.
\begin{algorithm}[MOOSE234\label{alg:vsvo}]

Let $m=4$. Given $\varepsilon$, $\tilde{\gamma}$, $\gamma$, $\{y_{n}, \cdots,y_{n+3}\}$, compute $y^{2}_{n+4},y^{3}_{n+4},y^{4}_{n+4}$ by solving
\begin{flalign*}
&\text{BDF3}  &&\sum_{j=1}^3 \left[\prod_{i=1}^{j-1}(t_{n+4} - t_{n+4-i}) \right]\delta^j y^{3} = f(t_{n+4},y^{3}_{n+4})\\
&\text{BDF3-Stab}  &&y^{2}_{n+4} = y^{3}_{n+4} + \frac{9}{125c^{(3)}_4}\delta^3 y^{3}\\
&\text{FBDF4}  &&y^{4}_{n+4} = y^{3}_{n+4} - \eta^{(4)}\delta^4 y^{3}.
\end{flalign*}
Put
\begin{alignat}{3}
&\text{Est}_2 = &&y^{3}_{n+4}-y^{2}_{n+4}\notag\\
&\text{Est}_3 = &&y^{4}_{n+4}-y^{3}_{n+4}\notag\\
&\text{Est}_4 = \bigg(&&\sum_{j=1}^4 \left[\prod_{i=1}^{j-1}(t_{n+4} - t_{n+4-i}) \right]\delta^j y^{4} - f(t_{n+4},y^{4}_{n+4})\bigg)\bigg/\bar{\alpha}^{(4)}_4 \label{eqn:lte_est_4}
\end{alignat}
Of the solutions that satisfy $|Est_i| < \varepsilon$, find $j$ that would allow a maximum step size to be taken.
\begin{equation}\label{eqn:order_selection}
j = \argmax_{i\in\{2,3,4\}} \left(\frac{\varepsilon}{|Est_i|}\right)^{\frac{1}{i+1}}
\end{equation}
Then set
\begin{equation}\label{eqn:stepsize_selection}
k_{n+4} = \gamma k_{n+3} \left(\frac{\varepsilon}{|Est_j|}\right)^{\frac{1}{j+1}}
\end{equation}
and
\begin{equation}\notag
y_{n+4}=y^{j}_{n+4}.
\end{equation}
If none satisfy the tolerance, set
\begin{equation}\notag
k_{n+3} \defeq \max_j \tilde{\gamma}k_{n+3}\left(\frac{\varepsilon}{|Est_j|}\right)^{\frac{1}{j+1}},
\end{equation}
and recompute the above steps.\\
\end{algorithm}

\eqref{eqn:stepsize_selection} is a standard formula for stepsize selection given an estimate for the local truncation error \cite{Shampine2005},\cite{griffiths},\cite{ahmad}.
$\gamma \leq 1$ is a safety factor used to keep the next step size from growing too fast to increase the chance that the next solution will be accepted. If a solution is rejected, $\tilde{\gamma} \leq \gamma$ is a second factor used to nudge the stepsize in a direction that will make the recomputed solution more likely to be accepted. We took $\gamma = 0.9$, and $\tilde{\gamma}=0.7$.

For step size control, we use the most popular criteria \cite{Shampine2005}, \cite{griffiths} which is to require the local truncation error be less than a user supplied tolerance $\varepsilon$. Another possible criterion for adapting, not tested herein, is to require the error per unit time interval be less than tolerance, $\text{LTE} <k_n\varepsilon$ \cite{conte},\cite{Shampine2005}. In our implementation, we add a common heuristic that the step size can change by no more than a factor of two at a time to avoid rejections \cite{ahmad}, although this may be overly cautious since factors as large as five have been used for adaptive BDF methods \cite{HAY2015151}.  Many other considerations for implementation and improvement of adaptive methods are discussed in the PhD thesis of Ahmad \cite{ahmad}.

The algorithm above is of variable order two through four, but different methods can be obtained by taking a max in \eqref{eqn:order_selection} over a subset of $\{2,3,4\}$, which is tested numerically in Section \ref{sec:numerical}.

\subsection{Concrete implementation of MOOSE234\label{sec:alt_implementation}}
We now state formulas for the coefficients used in the algorithm with stepsize ratios $\tau_n = \frac{k_{n}}{k_{n-1}}$ rather than divided differences. We will fix $\bar{k}_n \defeq k_n$. Then BDF$p$ is written 
\begin{equation}\notag
\frac{1}{k_{n+3}}\left( \sum_{i=0}^p \alpha^{(p)}_i y_{n+i} \right)= f(t_{n+p},y_{n+p})
\end{equation} The coefficients for this are well known \cite{wang2011}. For completeness, we include the coefficients for variable stepsize BDF3 and BDF4 in this section.  We will also show how the time filter coefficients can be expressed in terms of the BDF coefficients.
\begin{flalign*}
&\text{BDF3}  && \frac{1}{k_{n+3}}\left(\alpha^{(3)}_4y^{3}_{n+4} + \sum_{i=1}^3 \alpha^{(3)}_i y_{n+i} \right)= f(t_{n+4},y^{3}_{n+4})\\
&\text{BDF3-Stab}  &&y^{2}_{n+4} = y^{3}_{n+4} +C_4 y^{3}_{n+4}+ \sum_{i=1}^3C_i y_{n+i}\\
&\text{FBDF4}  &&y^{4}_{n+4} = y^{3}_{n+4} + D_4 y^{3}_{n+4}+ \sum_{i=0}^3D_i y_{n+i}
\end{flalign*}
Put
\begin{alignat}{3}
&\text{Est}_2 = &&y^{3}_{n+4}-y^{2}_{n+4}\notag\\
&\text{Est}_3 = &&y^{4}_{n+4}-y^{3}_{n+4}\notag\\
&\text{Est}_4 = \bigg(&&\alpha^{(4)}_4y^{4}_{n+4} + \sum_{i=1}^3 \alpha^{(4)}_i y_{n+i} - k_{n+3}f(t_{n+4},y^{4}_{n+4}) \bigg)\bigg/\alpha^{(4)}_4 \label{eqn:lte_est_4_explicit}
\end{alignat}
The coefficients for below BDF3, below, are padded with an extra zero to make the formulas for the filter coefficients more clear. These coefficients were generated using Mathematica.\\
\noindent
\textbf{BDF3 coefficients ($\alpha^{(3)}_i$)}
\begin{align*}
&\alpha^{(3)}_0 = 0\\
&\alpha^{(3)}_1 =-\frac{\tau _ {n+2}^3 \tau _ {n+3}^2 \left(1+\tau _{n+3}\right)}{\left(1+\tau _{n+2}\right) \left(1+\tau _ {n+2} \left(1+\tau_{n+3}\right)\right)}\\
&\alpha^{(3)}_2 =\tau _ {n+2} \tau _ {n+3}^2+\frac{\tau _ {n+3}^2}{1+\tau _{n+3}}\\
&\alpha^{(3)}_3 = -1-\tau _{n+3}-\frac{\tau _ {n+2} \tau _ {n+3} \left(1+\tau _{n+3}\right)}{1+\tau _{n+2}}\\
&\alpha^{(3)}_4 = 1+\frac{\tau _{n+3}}{1+\tau _{n+3}}+\frac{\tau _ {n+2} \tau_{n+3}}{1+\tau _ {n+2} \left(1+\tau _{n+3}\right)}\\
\end{align*}
\textbf{BDF4 coefficients ($\alpha^{(4)}_i$)}
\begin{align*}
\alpha^{(4)}_0 =& \frac{\tau _ {n+1}^4 \tau _ {n+2}^3 \tau _ {n+3}^2 \left(1+\tau_{n+3}\right) \left(1+\tau _ {n+2} \left(1+\tau _{n+3}\right)\right)}{\left(1+\tau _{n+1}\right) \left(1+\tau _ {n+1} \left(1+\tau_{n+2}\right)\right) \left(1+\tau _ {n+1} \left(1+\tau _ {n+2} 
\left(1+\tau _{n+3}\right)\right)\right)}\\
\alpha^{(4)}_1 =&-\frac{\tau _ {n+1} \tau _ {n+2}^3 \tau _ {n+3}^2 \left(1+\tau _{n+3}\
\right)}{1+\tau _{n+2}}-\frac{\tau _ {n+2}^3 \tau _ {n+3}^2 \left(1+\
\tau _{n+3}\right)}{\left(1+\tau _{n+2}\right) \left(1+\tau _ {n+2} \
\left(1+\tau _{n+3}\right)\right)}\\
\alpha^{(4)}_2 =&\tau _ {n+2} \tau _ {n+3}^2+\frac{\tau _ {n+3}^2}{1+\tau _{n+3}}+\frac{\tau _ {n+1} \tau _ {n+2} \tau _ {n+3}^2 \left(1+\tau _ {n+2} \left(1+\tau _{n+3}\right)\right)}{1+\tau _{n+1}}\\
\alpha^{(4)}_3 =& -1-\tau _{n+3}-\frac{\tau _ {n+2} \tau _ {n+3} \left(1+\tau _{n+3}\right)}{1+\tau _{n+2}}\\
&-\frac{\tau _ {n+1} \tau _ {n+2} \tau _ {n+3} \
\left(1+\tau _{n+3}\right) \left(1+\tau _ {n+2} \left(1+\tau _{n+3}\right)\right)}{\left(1+\tau _{n+2}\right) \left(1+\tau _ {n+1} \
\left(1+\tau _{n+2}\right)\right)}\\
\alpha^{(4)}_4 =& 1+\tau _ {n+3} \left(\frac{1}{1+\tau _{n+3}}+\frac{\tau _{n+2}}{1+\tau _ {n+2} \left(1+\tau _{n+3}\right)}\right)+\frac{\tau _ {n+1} \tau _ {n+2} \tau _{n+3}}{1+\tau _ {n+1} \left(1+\tau _ {n+2} \left(1+\tau _{n+3}\right)\right)}
\end{align*}
\noindent
\textbf{BDF3-Stab coefficients ($C_i$)}\\

\begin{align*}
C_0 = &0\\
C_1 =&-\frac{\mu  \tau _ {n+2}^2 \tau _ {n+3} \left(1+\tau _{n+3}\right)}{1+\tau _{n+2}}\\
C_2 = & \mu  \tau _ {n+3} \left(1+\tau _ {n+2} \left(1+\tau _{n+3}\right)\right)\\
C_3 = &-\frac{\mu  \left(1+\tau _{n+3}\right) \left(1+\tau _ {n+2} \left(1+\tau _{n+3}\right)\right)}{1+\tau _{n+2}}\\
C_4 = &\mu
\end{align*}
\noindent
\textbf{FDF4 filter coefficients ($D_i$)}
$$\gamma_i = \frac{\alpha^{(3)}_i-\alpha^{(4)}_i}{\alpha^{(3)}_4}$$
\begin{align*}
D_4 &= \frac{\gamma_4}{1-\gamma_4}\\
D_i &= \gamma_i (1+D_4) \hspace{10mm} \text{for }0 \leq i \leq 3. 
\end{align*}

The time filter coefficients $D_i$ for FBDF4 can be implemented with knowledge of the BDF3 and BDF4 coefficients alone. We now show the derivation of the formulas for $D_i$. The goal as usual is to eliminate the time filter step to yield an equivalent method. We start with the equation for the time filter omitting the super script 4,
\begin{equation}\notag
y_{n+4} = y^{3}_{n+4} + D_4 y^{3}_{n+4}+ \sum_{i=0}^3D_i y_{n+i}
\end{equation}
Add $D_4y^{4}_{n+4}$ to both sides,
\begin{equation}\notag
y_{n+4} + D_4y^{4}_{n+4}= y^{3}_{n+4} + D_4 y^{3}_{n+4}+ \sum_{i=0}^4D_i y_{n+i}.
\end{equation}
Solving for $y^{3}_{n+4}$,
\begin{equation}\notag
y^{3}_{n+4} = y_{n+4} - \frac{1}{1+D_4}\sum_{i=0}^4D_i y_{n+i}.
\end{equation}
Substituting this into the BDF3 step yields the equivalent method, the \emph{left hand} side of which is
\begin{equation}
\frac{1}{k_{n+3}}\left(\sum_{i=1}^4 \alpha^{(3)}_i y_{n+i} - \frac{\alpha^{(3)}_4}{1+D_4}\sum_{i=0}^4D_i y_{n+i}\right).
\end{equation}
Recall that the time filter was chosen so that the left hand side of the induced FBDF4 method is equal to the left hand side of BDF4. Thus, $D_i$ must satisfy
\begin{equation}\notag
\alpha^{(3)}_i - \frac{\alpha^{(3)}_4}{1+D_4}D_i = \alpha^{(4)}_i
\end{equation}
for $0\leq i \leq 4$. This completely determines $D_i$. Let $$\gamma_i = \frac{\alpha^{(3)}_i-\alpha^{(4)}_i}{\alpha^{(3)}_4}.$$ Solving the $i=4$ first for $D_4$, then substituting this into the other equations gives the coefficients as stated above.
\begin{remark}These formulas are easily generalized to the other FBDF methods. Given the coefficients of BDF$p$ and BDF$p$+1, similarly compact formulas for the filter coefficients can be derived following the same steps above.
\end{remark}
\section{Applications to Nonlinear Evolution Equations\label{sec:numerical}}
MOOSE234 is easily implemented for nonlinear evolution equations with decreased cost, increased assured accuracy and thereby increased predictive power. We give one test on a highly stiff and fluctuating ODE (the Van der Pol oscillator), and two tests for the Navier-Stokes equations, which describe a phenomena for which predictive accuracy is important and where memory limitations, accuracy requirements, cognitive and computational complexity are often in competition. 

The Van der Pol test is given in Section \ref{sec:van}. The NSE and the spatial discretization used herein are defined in Section \ref{sec:numerical_formulation}. In Section \ref{sec:numerical_tg_constant} , the constant stepsize, constant order methods are applied to a Taylor-Green vortex array. This is a common benchmark problem in CFD such as \cite{chorin1968}, and others. The VSVO method is tested for the same problem in Section \ref{sec:numerical_tg_adaptive}. 

The methods we develop solve one fully implicit step, and then apply time filters to achieve higher order. We measure error versus work by the number of time steps taken plus the number of rejected solution in Section \ref{sec:van}, and by compute time in Section \ref{sec:numerical_tg_adaptive}. We test methods of different orders (2,23,234,3,34,4) by restricting the approximations that  MOOSE234 is allowed to select. The method of order three is simply adaptive BDF3, 23 is adaptive BDF3 and BDF3-Stab, etc. If the method does not include 4, we do not evaluate the error estimator for FBDF4 so that it does not artificially inflate the runtime of the lower order methods.
\subsection{Van der Pol Oscillator\label{sec:van}}
In this section, we test the methods on the Van der Pol oscillator, a common benchmark problem for stiff ODE integrators. 
\begin{align*}
&y_1'=y_2\\
&y_2'=\bar{\mu}(1-y_1^2)y_2-y_1
\end{align*}
with $\bar{\mu}=1000$.
We compute relative errors at $t=3000$ by comparing with a reference solution from MATLAB's ode15s with an absolute tolerance of 1e-16, and a relative tolerance of 3e-14. 

The error vs total work (number of steps taken plus rejected steps) is shown in Figure \ref{fig:variable_error}, and clearly shows the higher order methods are most efficient for this problem.  We tested many combinations of orders (2,23,234,3,34,4) to verify that the higher order methods reduced the total amount of work, although we do not plot the results of all these combinations. Specific to this test, we note that 23 performed essentially the same as adaptive BDF3, and 34 was essentially the same as MOOSE234. Adaptive FBDF4 appears to perform the best, although it does not have as obvious of a trend. However, other tests such as the one performed in Section \ref{sec:numerical_tg_adaptive} do show a notable increase in efficiency using the full MOOSE234 versus using a subset of the available orders. 

The stepsize and order evolution of MOOSE234 is shown in Figure \ref{fig:vanderpolstepsize}.\\ MOOSE234 chooses BDF3 as the approximation in the relatively flat regions, and switches to the other orders in the region of rapid transition.

\begin{figure}
\begin{center}
\includegraphics[width=0.9\textwidth]{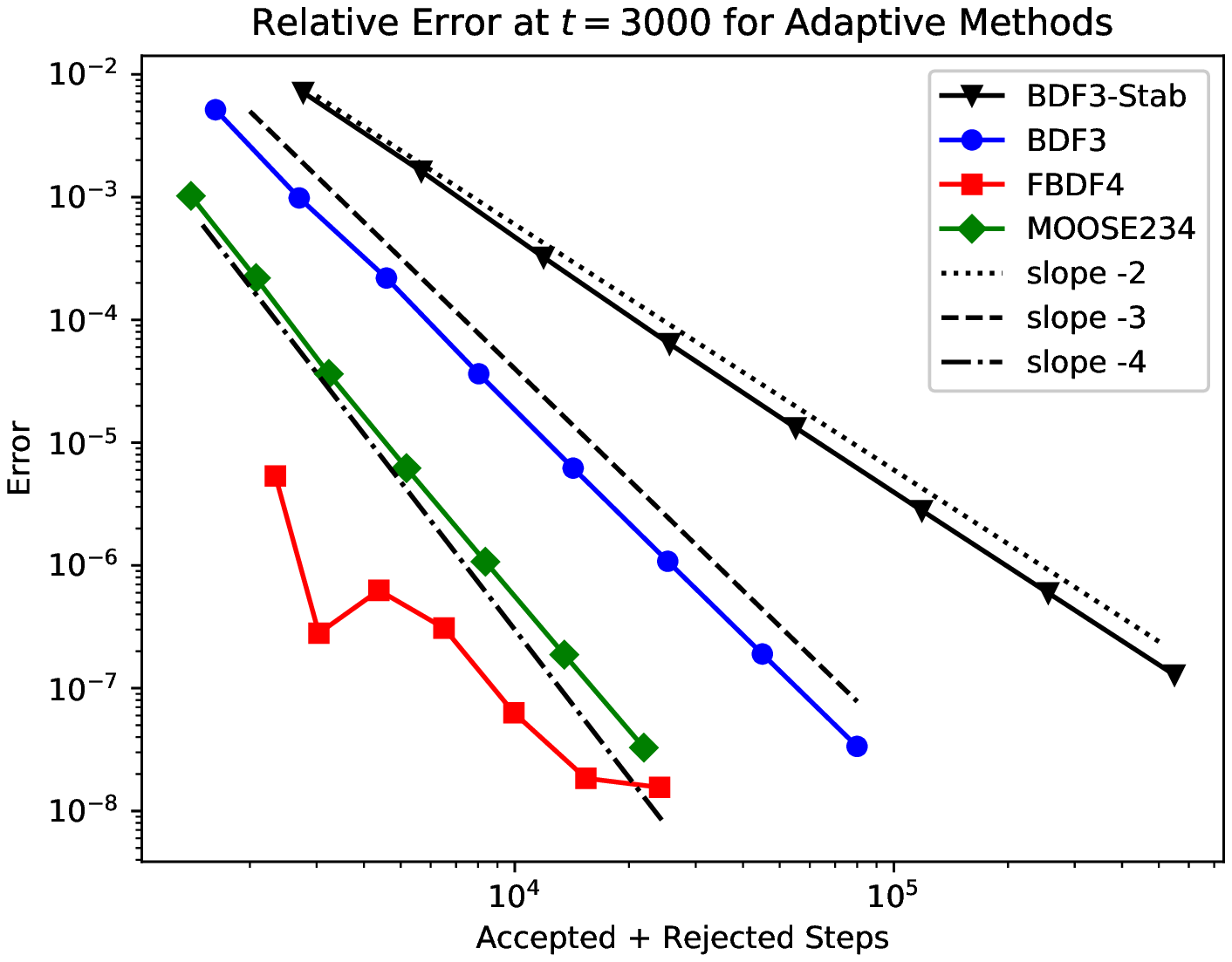}
\caption{ \label{fig:vanderpolconvergence}}
\end{center} 
\end{figure}
\begin{figure}
\begin{center}
\includegraphics[width=0.9\textwidth]{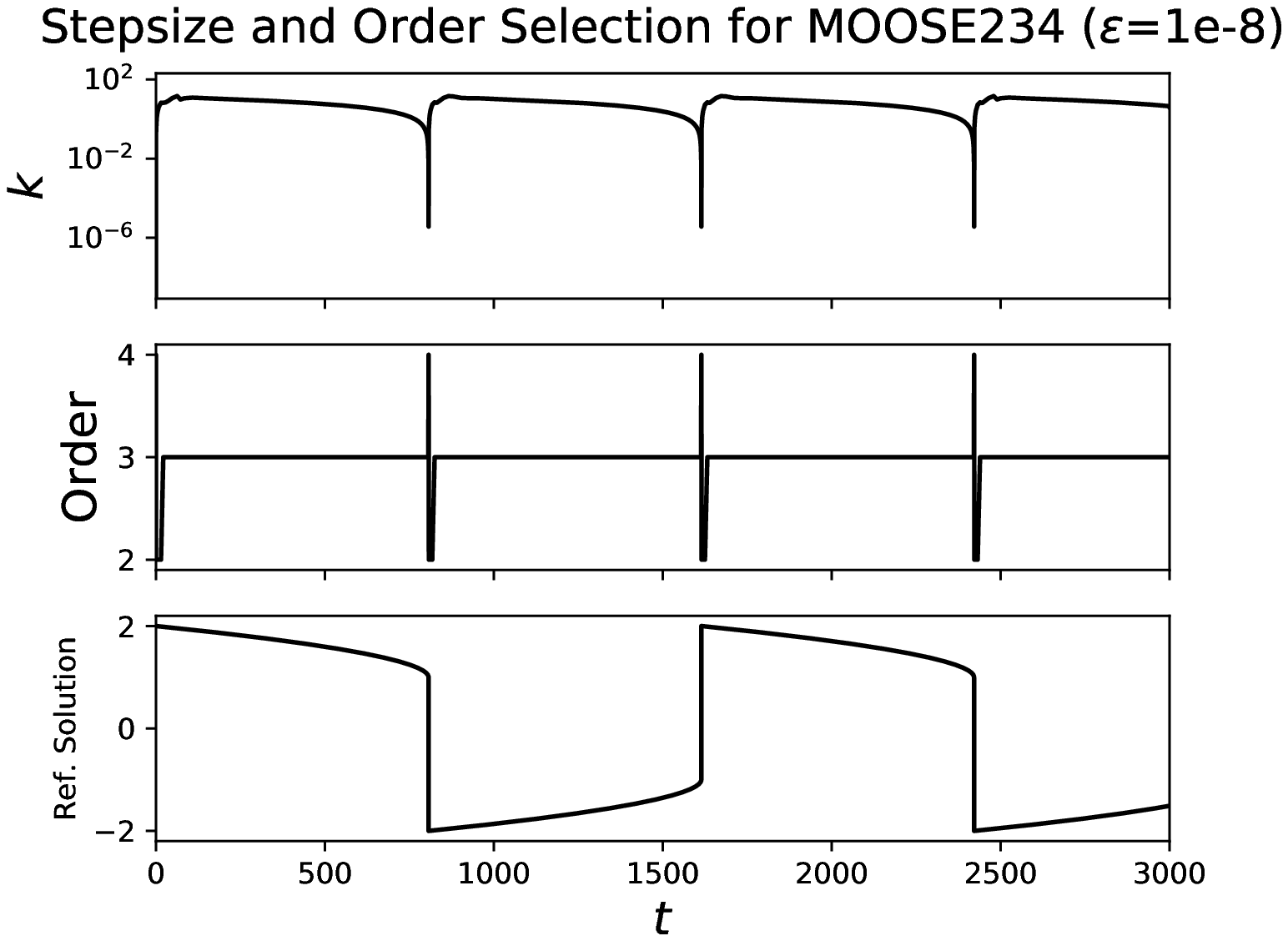}
\caption{ \label{fig:vanderpolstepsize}}
\end{center} 
\end{figure}

\subsection{Finite Element Formulation\label{sec:numerical_formulation}}
Given the domain $%
\Omega \subset \mathbb{R}^{d}$ ($d=2,3$), consider the problem 
\begin{align*}
u_{t}+u\cdot \nabla {u}-\nu \Delta {u}+\nabla {p}=f~& \text{and}~\nabla
\cdot u=0~\text{in}~\Omega \times (0,T] \\
u=0~\text{on}~\partial \Omega ~& \text{and}~u(x,0)=u_{0}(x).
\end{align*}
To discretize in space, let $(X^h,Q^h)$ be a finite element pair satisfying the $LBB_h$ condition. We suppress the spatial discretization on velocity and pressure to avoid excessive super and subscripts. Define the explicitly skew-symmeterized trilinear form
\begin{equation}\notag
b^*(u,v,w) = (u \cdot \nabla v,w) + \frac{1}{2}((\nabla \cdot u)v,w)
\end{equation}
The fully discrete BDF3 problem is as follows. Find $(u^{3}_{n+4},p_{n+4}) \in (X^h,Q^h)$ such that for all $(v^h,q^h) \in (X^h,Q^h)$,
\begin{gather*}
\left(\bar{\alpha}^{(3)}_4u^{3}_{n+4} +\sum_{j=1}^3\bar{\alpha}^{(3)}_ju_{n+j},v^h\right) + b^*(u^{3}_{n+4},u^{3}_{n+4},v^h) + \nu (\nabla u^{3}_{n+4},\nabla v^h) \\-(p_{n+4},\nabla \cdot v^h) = (f(t_{n+4}),v^h) \\
(\nabla \cdot u^{3}_{n+4}, q_h) = 0.
\end{gather*}
This method for constant stepsize was analyzed in \cite{bakerBdfthree82} with the nonlinearity extrapolated with $$b^*(u_{n+4},u_{n+4},v_h)\approx b^*(3u_{n+3}-3u_{n+2}+u_{n+1},3u_{n+3}-3u_{n+2}+u_{n+1},v_h)$$ so that one linear solve was required each step.
We linearize the problem as follows. Let 
\begin{equation}\notag
\hat{u}^{3}_{n+4} = -\frac{1}{c^{(4)}_4}\sum_{i=0}^3c^{(4)}_iu_{n+i}
\end{equation}
which is a fourth order extrapolation of $u^{3}_{n+4}$ to preserve the order of consistency of FBDF4. Then the linearly implicit (sometimes called semi-implicit) method is as follows. Find $(u^{3}_{n+4},p_{n+4}) \in (X^h,Q^h)$ such that for all $(v^h,q^h) \in (X^h,Q^h)$,
\begin{gather*}
\left(\bar{\alpha}^{(3)}_4u^{3}_{n+4} +\sum_{j=1}^3\bar{\alpha}^{(3)}_ju_{n+j},v^h\right) + b^*(\hat{u}^{3}_{n+4},u^{3}_{n+4},v^h) + \nu (\nabla u^{3}_{n+4},\nabla v^h) \\ -(p_{n+4},\nabla \cdot v^h) = (f(t_{n+4}),v^h) \\
(\nabla \cdot u^{3}_{n+4}, q_h) = 0.
\end{gather*}
Linearly implicit methods are a common way to reduce the computational complexity of time stepping nonlinear parabolic problems \cite{douglas}\cite{akrivis2015}. The idea of Baker \cite{bakerunpublished} to treat the convective term in Galerkin approximations of Navier-Stokes this way while preserving skew-symmetry has a long history of use and expansions \cite[p. 185]{girault} \cite{ingram} \cite{ravindran2015} \cite{jiang2017}\cite{akbas}.

Pressure is not a dynamic variable, and only the pressure at the current time level is required so that applying the time filters to pressure will not effect the computed velocity solution. For these reasons, we choose not to filter it for these tests. Therefore, $Est_2$, $Est_3$, and $Est_4$ are only estimates of the temporal velocity error. $|Est_i| = \norm{Est_i}$ is the $L^2(\Omega)$ norm. Whether or not the pressure should be filtered is an open problem.

The differentiation error for the FBDF4 is estimated by $Est_4$, where $Est_4$ is the finite element discretization of \eqref{eqn:error_est_fbdf_difference}, and is the solution of
\begin{gather*}
\bar{\alpha}^{(4)}_4(Est_4,v^h) = (\bar{\alpha}^{(4)}_4u^{4}_{n+4} +\sum_{j=0}^3\bar{\alpha}^{(4)}_ju_{n+j},v^h)
+b^*(u^{4}_{n+4},u^{4}_{n+4},v^h)\\ - (p_{n+4},\nabla \cdot v^h) - (f(t_{n+4}),v^h) 
 \hspace{10mm}\text{  for all }v^h \in X^h.
\end{gather*}
This is the only non-embedded error estimator, and since it is a mass matrix with order one condition number and narrow band width, does not add significantly to the computational complexity. In our tests in Section \ref{sec:numerical_tg_adaptive} with 495,000 degrees of freedom, this system can be solved with about 10 iterations of the conjugate gradient method within a relative tolerance  of 1e-6; this takes about 0.1 seconds on a desktop with a four core Intel i7 7700k cpu. The time taken to solve this system is included in the timing tests in Figure \ref{fig:variable_error}.
All tests were performed with FEniCS \cite{fenics}.

In the adaptive tests, the stepsize ratios were limited to a maximum of two and a minimum of one half, which is a common heuristic in variable stepsize methods \cite{ahmad}. All tests were performed on a square periodic domain using P3/P2 Lagrange elements with 150 elements per edge of the square resulting in 495k degrees of freedom.
\subsection{Constant Stepsize\label{sec:numerical_tg_constant}}
We test the case of constant step size on a Taylor-Green vortex array, a common benchmark problem in CFD.

We took $\Omega$ to be the periodic box with sides of length $2\pi$, and $\nu = 1$. Define $F(t) = e^{- 2\nu t}$. An exact solution is given by
\begin{align*}\notag
u(x,y;t) = F(t)(\cos x \sin y, -\cos y \sin x)\\
p(x,y;t) = -\frac{1}{4} F(t)^2 (\cos 2x + \cos 2y).
\end{align*}
Figures \ref{fig:velConvergence} and \ref{fig:presConvergence} show the relative l2L2 velocity and pressure errors for different stepsize $k$. We achieve convergence rates in time for the velocity predicted by the ODE theory. Pressure errors are reduced in the higher order methods even though they are not filtered, although a much finer mesh may be required to see the same orders of convergence as the velocity. The best treatment of the pressure is still under investigation. 

\begin{figure}
\centering
\begin{subfigure}{0.45\linewidth}
   \centering
   \includegraphics[width = 1\linewidth]{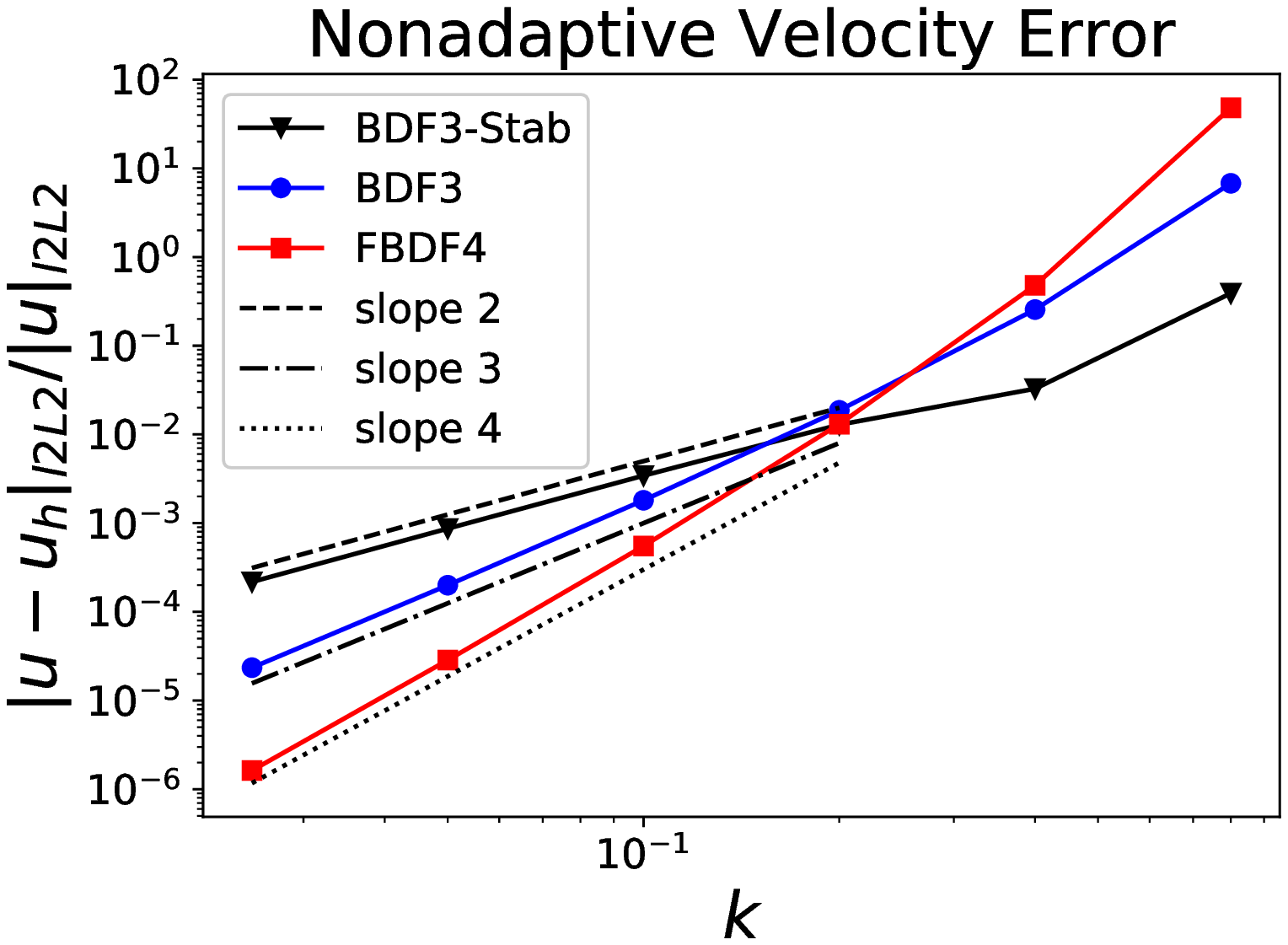}
   \caption{\label{fig:velConvergence}}
\end{subfigure}
\begin{subfigure}{0.45\linewidth}
   \centering
   \includegraphics[width = 1\linewidth]{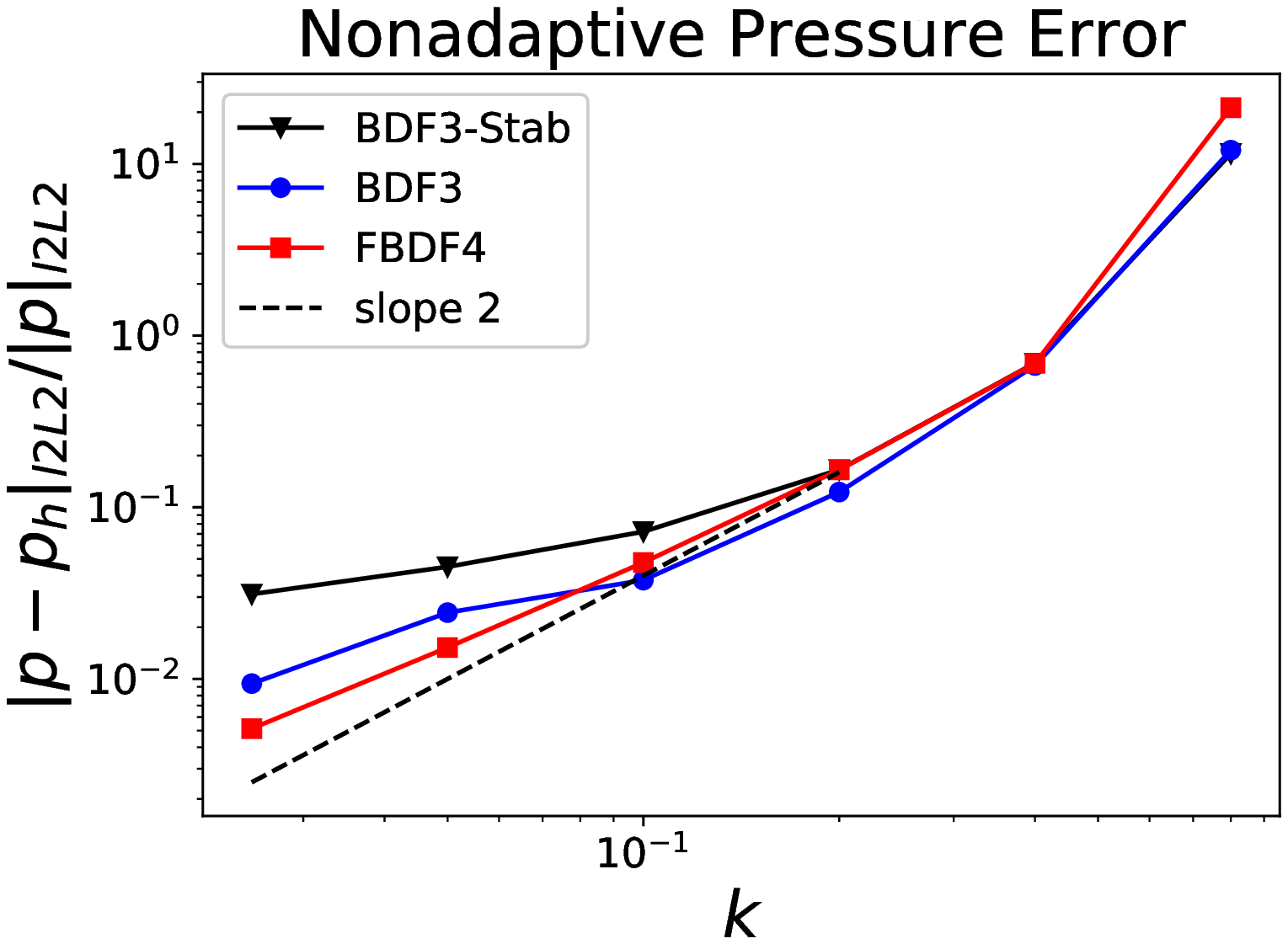}
   \caption{\label{fig:presConvergence}}
\end{subfigure}
\caption{Velocity converges at the predicted rates. FBDF4 improves convergence of the pressure.\selectfont{}}
\end{figure}

\subsection{Variable Stepsize Variable Order\label{sec:numerical_tg_adaptive}}
In this test, we allow the method to adapt, and run the above problem to a final time of $T=10$. The same mesh from the constant stepsize test was used. All tests were initialized with exact solutions that were $k = $1E-3 apart. The safety factor $\gamma=0.9$, and the safety factor used in the event of a rejected step was $\tilde{\gamma}=0.7$.

We tested many combinations of orders (2,23,234,3,34,4) to verify that variable order is necessary for an improvement in execution time. We do not show the results of all these combinations in the plots for clarity, but we do note that the method of order 23 is slightly more efficient than adaptive BDF3. The method of order 34 performed better than FBDF4, but slightly worse than MOOSE234 for larger tolerances. Each test was timed starting at the outer time stepping loop of the program, and ending after the final time step. Various $\varepsilon$ were tested from 1e-1 to 1e-8. Figure \ref{fig:variable_error} shows the amount of time in seconds each method required to run to completion for different tolerances versus the relative l2L2 errors, with the tolerances decreasing from left to right.

For the smallest tolerances, we clearly see that the higher order methods are the most efficient with the full MOOSE234 method performing the best. MOOSE234 is about three times faster than adaptive BDF3 for the final tolerance.

\begin{figure}
\begin{center}
\includegraphics[width=0.9\textwidth]{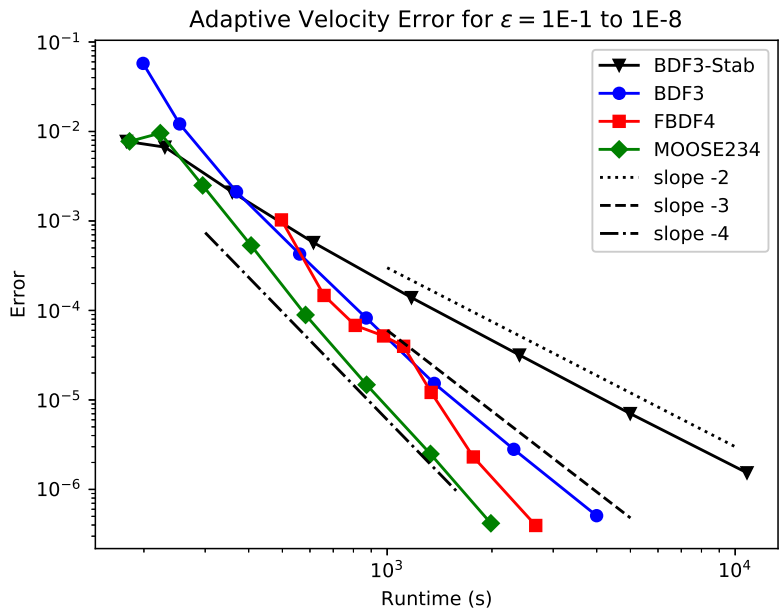}
\includegraphics[width=0.9\textwidth]{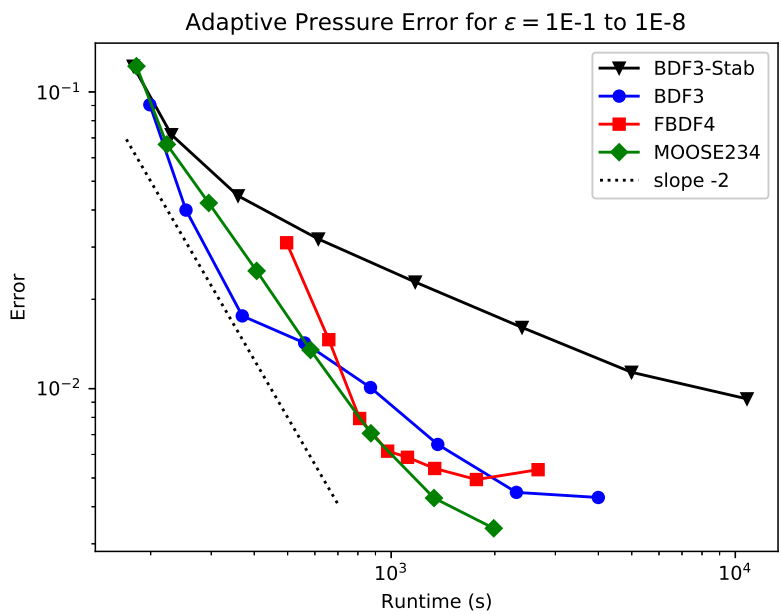}
\caption{The greatest accuracy for the least compute time is from the higher order methods.\label{fig:variable_error}}
\end{center} 
\end{figure}

\section{Conclusion\label{sec:conclusion}}

We present MOOSE234, a new stiff VSVO solver orders two, three, and four. The computational complexity is comparable to BDF3. In our tests on the Van der Pol oscillator and a standard spatial discretization of the Navier-Stokes equations, the VSVO methods of higher order give the most accurate approximations at least cost. We also developed FBDF$p$+1 of orders two through six, which uses computationally inexpensive time filters to increase the order all variable stepsize BDF$p$ methods with $p\leq 5$ by one. 

Many open problems remain. There may be a more general order barrier for filtering OLMs and LMMs. Linearly implicit (tested herein for Navier-Stokes) and Implicit-Explicit versions need systematic development. Error analysis of the fully discrete method for NSE, and a deeper understanding of the pressure error are needed. There may exist more optimal $G$-Stabilizing filters for BDF3. The idea of constructing $G$-stabilizing time filters can be applied to higher order BDF methods, and other multistep methods. MOOSE234 can be applied to other complex nonlinear applications. FBDF$p$+1 besides FBDF4 can be embedded in other methods.

\textbf{Acknowledgements}
The authors are grateful to Michael Neilan and Catalin Trenchea for useful discussions.

This research was partially supported by NSF grants DMS1522267, 1817542, CBET 1609120,  NSFC grant 11571274 and China Scholarship Council grant \\201706280334.
\section{Appendix}
\subsection{Code to calculate BDF coefficients\label{sec:psuedocode}}
\begin{algorithm}[BDF and Filter Coefficients]\label{alg:with_div_differences}
$\boldsymbol{d}^j= \boldsymbol{e}_{m+1-j} \in \mathbb{R}^{m+1}$ for $j \in \{0,1,\cdots,m\}$.

\textbf{\emph{FUNCTION}} BACKDIFF($t_{n},...,t_{n+m}$)

\qquad $(c^{(0)}_0,c^{(0)}_1,\cdots,c^{(0)}_m) = \boldsymbol{d}^0$

\qquad \textbf{\emph{FOR}} \emph{q = 1:m} 

\qquad \qquad \textbf{\emph{FOR}} \emph{j = 0:m-q}

\qquad \qquad \qquad $\boldsymbol{d}^j = (\boldsymbol{d}^j - \boldsymbol{d}^{j+1})/(t_{n+m-j} -  t_{n+m-q-j})$

\qquad \qquad \textbf{\emph{END FOR}}

\qquad \qquad $(c^{(q)}_0,c^{(q)}_1,\cdots,c^{(q)}_m) = \boldsymbol{d}^0$

\qquad \textbf{\emph{END FOR}}

\textbf{\emph{RETURN}} $\{c^{(j)}_i\}_{i,j=0}^m$

\textbf{//The function below calculates the coefficients for BDF$p$,}

\textbf{//$\eta^{(p+1)}$, and the coefficients of the divided differences $c^{(j)}_i$ }

\textbf{\emph{FUNCTION}} BDFANDFILTCOEFF($t_{n},...,t_{n+m}$,p)

\qquad $\{c^{(j)}_i\}_{i,j=0}^m$ = BACKDIFF($t_{n},...,t_{n+m}$)

\qquad $\eta^{(p+1)} = \left(\prod_{i=1}^p(t_{n+m}-t_{n+m-i})\right)/\left(\sum_{j=1}^{p+1}(t_{n+m}-t_{n+m-j})^{-1}\right)$ 

\qquad \textbf{\emph{FOR}} \emph{k = m-p:m} 

\qquad \qquad $\bar{\alpha}_k = \sum_{j=1}^p \left[\prod_{i=1}^{j-1}(t_{n+m} - t_{n+m-i}) \right]c^{(j)}_k$

\qquad \textbf{\emph{END FOR}}

\textbf{\emph{RETURN}} $\{\bar{\alpha}_k\}_{k=m-p}^m, \eta^{(p+1)}, \emph{ and } \{c^{(j)}_i\}_{i,j=0}^m$

\end{algorithm}
\subsection{Python Implementation}

We also include an implementation \\of the above functions in Python.
The ``\textbackslash '' character is the continue line command.

\begin{lstlisting}
import numpy as np

def first_difference(T,Y):
	"""
	T is a vector of times, greatest to least
	Y is a vector older differences, [d(j),d(j+1)]
	See algorthm in paper
	"""
	return (Y[0]-Y[1])/(T[0] - T[1])

def backward_differences(T):
	"""
	Generate the divided difference coefficients.
	T is a vector of times from least to greatest, e.g.
	T = [t_n,t_n+1,....,t_n+m]
	"""
	numOfTimes = len(T)
	#the number of steps in the method
	m = numOfTimes - 1
	#generate the initial differences, which
	#is just the standard basis.
	D = np.array([ [np.float64((i+1)==(numOfTimes-j))\
	  for i in xrange(numOfTimes)] \
	    for j in xrange(numOfTimes)])
	differences = np.zeros_like(D)
	differences[0] = D[0]
	
	for q in xrange(1,numOfTimes):
	 for j in xrange(numOfTimes - q):
		D[j] = first_difference\
			([T[m-j],T[m-j-q]],[D[j],D[j+1]])
		differences[q] = D[0]
	return differences

def bdf_coefficients_and_differences(T,order):
	differences = backward_differences(T)
	m = len(T)-1
	#calculate filter coefficient for increasing order
	eta = np.prod([T[m]-T[m-i] \
	 for i in xrange(1,m)])/np.sum(1./(T[m] - T[m-j]) \ 
	  for j in xrange(1,m+1))

	return [np.sum(np.prod([T[m]-T[m-i] \
	 for i in xrange(1,j)])*differences[j] \
	  for j in xrange(1,order+1)), differences,eta]	
	
\end{lstlisting}

\normalsize

\subsection{Coefficients of $G$ matrix}
\begin{eqnarray}
\begin{aligned}\label{eqn:bdf1}
&g33 =(\mu + 3^{\frac{1}{2}}((7\mu - 1)(6\mu - 5)(14\mu - 1)(\mu + 1)^5)^{\frac{1}{2}} \\
&+ 41\mu^2 + 83\mu^3 + 42\mu^4 + 1)/(44(\mu^4 + 4\mu^3 + 6\mu^2 + 4\mu + 1)) \\
&+ (- 42\mu^2 + \mu + 21)/(22(\mu^2 + 2\mu + 1));
\end{aligned}
\end{eqnarray}

\begin{eqnarray}
\begin{aligned}\label{eqn:bdf2}
&g32 =(42\mu^2 + 13\mu - 7)/(11(\mu^2 + 2\mu + 1)) - (\mu + 3^{\frac{1}{2}}((7\mu - 1)\\
&(6\mu - 5)(14\mu - 1)(\mu + 1)^5)^{\frac{1}{2}} + 41\mu^2 + 83\mu^3 + 42\mu^4 + 1)\\
&/(22(\mu^4 + 4\mu^3 + 6\mu^2 + 4\mu + 1));
\end{aligned}
\end{eqnarray}

\begin{eqnarray}
\begin{aligned}\label{eqn:bdf3}
&g31 =(\mu + 3^{\frac{1}{2}}((7\mu - 1)(6\mu - 5)(14\mu - 1)(\mu + 1)^5)^{\frac{1}{2}} + 41\mu^2 \\
&+ 83\mu^3 + 42\mu^4 + 1)/(44(\mu^4 + 4\mu^3 + 6\mu^2 + 4\mu + 1)) - (21\mu^2 \\
 &+8\mu - 2)/(11(\mu^2 + 2\mu + 1));
\end{aligned}
\end{eqnarray}

\begin{eqnarray}
\begin{aligned}\label{eqn:bdf4}
&g22 =(\mu + 3^{\frac{1}{2}}((7\mu - 1)(6\mu - 5)(14\mu - 1)(\mu + 1)^5)^{\frac{1}{2}} + 41\mu^2\\
& + 83\mu^3 + 42\mu^4 + 1)/(11(\mu^4 + 4\mu^3 + 6\mu^2 + 4\mu + 1)) - (120\\
& \mu^2+ 23\mu - 9)/(22(\mu^2 + 2\mu + 1));
\end{aligned}
\end{eqnarray}

\begin{eqnarray}
\begin{aligned}\label{eqn:bdf5}
&g21 =(51\mu^2 + 5\mu - 2)/(22(\mu^2 + 2\mu + 1)) - (\mu + 3^{\frac{1}{2}}((7\mu - 1)\\
&(6\mu - 5)(14\mu - 1)(\mu + 1)^5)^{\frac{1}{2}} + 41\mu^2 + 83\mu^3 + 42\mu^4 + 1)\\
&/(22(\mu^4 + 4\mu^3 + 6\mu^2 + 4\mu + 1));
\end{aligned}
\end{eqnarray}

\begin{eqnarray}
\begin{aligned}\label{eqn:bdf6}
&g11 =(- 9\mu^2 + 2\mu)/(11(\mu^2 + 2\mu + 1)) + (\mu + 3^{\frac{1}{2}}((7\mu - 1)\\
&(6\mu - 5)(14\mu - 1)(\mu + 1)^5)^{\frac{1}{2}} + 41\mu^2 + 83\mu^3 + 42\mu^4 + 1)\\
&/(44(\mu^4 + 4\mu^3 + 6\mu^2 + 4\mu + 1));
\end{aligned}
\end{eqnarray}

\begin{eqnarray}
\begin{aligned}\label{eqn:bdf7}
& a3 =(22*\mu^2((\mu + 3^{\frac{1}{2}}((7\mu - 1)(6\mu - 5)(14\mu - 1)(\mu + 1)^5)^{\frac{1}{2}}\\
& + 41\mu^2 + 83\mu^3 + 42\mu^4 + 1)/(44(\mu^4 + 4\mu^3 + 6\mu^2 + 4\mu + 1)))^{\frac{3}{2}}\\
& - 42\mu^2((\mu + 3^{\frac{1}{2}}((7\mu - 1)(6\mu - 5)(14\mu - 1)(\mu + 1)^5)^{\frac{1}{2}}\\
& + 41\mu^2 + 83\mu^3 + 42\mu^4 + 1)/(44(\mu^4 + 4\mu^3 + 6\mu^2 + 4\mu + 1)))^{\frac{1}{2}} \\
&+ \mu((\mu + 3^{\frac{1}{2}}((7\mu - 1)(6\mu - 5)(14\mu - 1)(\mu + 1)^5)^{\frac{1}{2}} \\
&+ 41\mu^2 + 83\mu^3 + 42\mu^4 + 1)/(44(\mu^4 + 4\mu^3 + 6\mu^2 + 4\mu + 1)))^{\frac{1}{2}}\\
& + 44\mu((\mu + 3^{\frac{1}{2}}((7\mu - 1)(6\mu - 5)(14\mu - 1)(\mu + 1)^5)^{\frac{1}{2}} \\
&+ 41\mu^2 + 83\mu^3 + 42\mu^4 + 1)/(44(\mu^4 + 4\mu^3 + 6\mu^2 + 4\mu + 1)))^{\frac{3}{2}}\\
& - ((\mu + 3^{\frac{1}{2}}((7\mu - 1)(6\mu - 5)(14\mu - 1)(\mu + 1)^5)^{\frac{1}{2}}\\
& + 41\mu^2 + 83\mu^3 + 42\mu^4 + 1)/(44(\mu^4 + 4\mu^3 + 6\mu^2 + 4\mu + 1)))^{\frac{1}{2}} \\
&+ 22((\mu + 3^{\frac{1}{2}}((7\mu - 1)(6\mu - 5)(14\mu - 1)(\mu + 1)^5)^{\frac{1}{2}} \\
&+ 41\mu^2 + 83\mu^3 + 42\mu^4 + 1)/(44(\mu^4 + 4\mu^3 + 6\mu^2 + 4\mu \\
&+ 1)))^{\frac{3}{2}}/(20\mu - 2)
;
\end{aligned}
\end{eqnarray}

\begin{eqnarray}
\begin{aligned}\label{eqn:bdf8}
&a2 =-(22\mu^2((\mu + 3^{\frac{1}{2}}((7\mu - 1)(6\mu - 5)(14\mu - 1)(\mu + 1)^5)^{\frac{1}{2}}\\
& + 41\mu^2 + 83\mu^3 + 42\mu^4 + 1)/(44(\mu^4 + 4\mu^3 + 6\mu^2 + 4\mu + 1)))^{\frac{3}{2}}\\
& - 42\mu^2((\mu + 3^{\frac{1}{2}}((7\mu - 1)(6\mu - 5)(14\mu - 1)(\mu + 1)^5)^{\frac{1}{2}}\\ 
&+ 41\mu^2 + 83\mu^3 + 42\mu^4 + 1)/(44(\mu^4 + 4\mu^3 + 6\mu^2 + 4\mu + 1)))^{\frac{1}{2}}\\
& + 11\mu((\mu + 3^{\frac{1}{2}}((7\mu - 1)(6\mu - 5)(14\mu - 1)(\mu + 1)^5)^{\frac{1}{2}}\\
& + 41\mu^2 + 83\mu^3 + 42\mu^4 + 1)/(44(\mu^4 + 4\mu^3 + 6\mu^2 + 4\mu + 1)))^{\frac{1}{2}}\\
& + 44\mu((\mu + 3^{\frac{1}{2}}((7\mu - 1)(6\mu - 5)(14\mu - 1)(\mu + 1)^5)^{\frac{1}{2}}\\
& + 41\mu^2 + 83\mu^3 + 42\mu^4 + 1)/(44(\mu^4 + 4\mu^3 + 6\mu^2 + 4\mu + 1)))^{\frac{3}{2}}\\
& - 2((\mu + 3^{\frac{1}{2}}((7\mu - 1)(6\mu - 5)(14\mu - 1)(\mu + 1)^5)^{\frac{1}{2}} \\
&+ 41\mu^2 + 83\mu^3 + 42\mu^4 + 1)/(44(\mu^4 + 4\mu^3 + 6\mu^2 + 4\mu + 1)))^{\frac{1}{2}}\\
& + 22((\mu + 3^{\frac{1}{2}}((7\mu - 1)(6\mu - 5)(14\mu - 1)(\mu + 1)^5)^{\frac{1}{2}} \\
&+ 41\mu^2 + 83\mu^3 + 42\mu^4 + 1)/(44(\mu^4 + 4\mu^3 + 6\mu^2 + 4\mu \\
&+ 1)))^{\frac{3}{2}}/(10\mu - 1)
;
\end{aligned}
\end{eqnarray}

\begin{eqnarray}
\begin{aligned}\label{eqn:bdf9}
&a1 =(22\mu^2((\mu + 3^{\frac{1}{2}}((7\mu - 1)(6\mu - 5)(14\mu - 1)(\mu + 1)^5)^{\frac{1}{2}}\\
& + 41\mu^2 + 83\mu^3 + 42\mu^4 + 1)/(44(\mu^4 + 4\mu^3 + 6\mu^2 + 4\mu + 1)))^{\frac{3}{2}}\\
&- 42\mu^2((\mu + 3^{\frac{1}{2}}((7\mu - 1)(6\mu - 5)(14\mu - 1)(\mu + 1)^5)^{\frac{1}{2}} \\
&+ 41\mu^2 + 83\mu^3 + 42\mu^4 + 1)/(44(\mu^4 + 4\mu^3 + 6\mu^2 + 4\mu + 1)))^{\frac{1}{2}} \\
&+ 41\mu((\mu + 3^{\frac{1}{2}}((7\mu - 1)(6\mu - 5)(14\mu - 1)(\mu + 1)^5)^{\frac{1}{2}} \\
&+ 41\mu^2 + 83\mu^3 + 42\mu^4 + 1)/(44(\mu^4 + 4\mu^3 + 6\mu^2 + 4\mu + 1)))^{\frac{1}{2}}\\
&+ 44\mu((\mu + 3^{\frac{1}{2}}((7\mu - 1)(6\mu - 5)(14\mu - 1)(\mu + 1)^5)^{\frac{1}{2}}\\ 
 &+ 41\mu^2 + 83\mu^3 + 42\mu^4 + 1)/(44(\mu^4 + 4\mu^3 + 6\mu^2 + 4\mu + 1)))^{\frac{3}{2}}\\ 
 &- 5((\mu + 3^{\frac{1}{2}}((7\mu - 1)(6\mu - 5)(14\mu - 1)(\mu + 1)^5)^{\frac{1}{2}} \\
 &+ 41\mu^2 + 83\mu^3 + 42\mu^4 + 1)/(44(\mu^4 + 4\mu^3 + 6\mu^2 + 4\mu + 1)))^{\frac{1}{2}}\\
  &+ 22((\mu + 3^{\frac{1}{2}}((7\mu - 1)(6\mu - 5)(14\mu - 1)(\mu + 1)^5)^{\frac{1}{2}}\\ 
  &+ 41\mu^2 + 83\mu^3 + 42\mu^4 + 1)/(44(\mu^4 + 4\mu^3 + 6\mu^2 + 4\mu \\
 & + 1)))^{\frac{3}{2}}/(20\mu - 2)
;
\end{aligned}
\end{eqnarray}

\begin{eqnarray}
\begin{aligned}\label{eqn:bdf10}
&a0 =-((\mu + 3^{\frac{1}{2}}((7\mu - 1)(6\mu - 5)(14\mu - 1)(\mu + 1)^5)^{\frac{1}{2}}+41  \\
&\mu^2+ 83\mu^3 + 42\mu^4 + 1)/(44(\mu^4 + 4\mu^3 + 6\mu^2 + 4\mu + 1)))^{\frac{1}{2}};
\end{aligned}
\end{eqnarray}

\bibliography{vsvo}

\begin{thebibliography}{10}

\bibitem{ahmad}
{\sc S.~Ahmad}, {\em Design and construction of software for general linear
  methods}, PhD thesis, Massey University, 2016.

\bibitem{akbas}
{\sc M.~Akbas, S.~Kaya, and L.~G. Rebholz}, {\em On the stability at all times
  of linearly extrapolated {BDF2} timestepping for multiphysics incompressible
  flow problems}, Numerical Methods for Partial Differential Equations, 33,
  pp.~999--1017.

\bibitem{akrivis2015}
{\sc G.~Akrivis and C.~Lubich}, {\em Fully implicit, linearly implicit and
  implicit---explicit backward difference formulae for quasi-linear parabolic
  equations}, Numer. Math., 131 (2015), pp.~713--735.

\bibitem{fenics}
{\sc M.~S. Aln{\ae}s, J.~Blechta, J.~Hake, A.~Johansson, B.~Kehlet, A.~Logg,
  C.~Richardson, J.~Ring, M.~E. Rognes, and G.~N. Wells}, {\em The {FE}ni{CS}
  project version 1.5}, Archive of Numerical Software, 3 (2015).

\bibitem{aluthge2016}
{\sc A.~Aluthge, S.~Sarra, and R.~Estep}, {\em Filtered leapfrog time
  integration with enhanced stability properties}, Journal of Applied
  Mathematics and Physics, 04 (2016), pp.~1354--1370.

\bibitem{asselin72}
{\sc R.~Asselin}, {\em Frequency filter for time integrations}, Monthly Weather
  Review, 100 (1972), pp.~487--490.

\bibitem{bakerunpublished}
{\sc G.~Baker}, {\em Galerkin approximations for the {N}avier-{Stokes}
  equations}, tech. rep., Harvard University, 1976.

\bibitem{bakerBdfthree82}
{\sc G.~A. Baker, V.~A. Dougalis, and O.~A. Karakashian}, {\em On a higher
  order accurate fully discrete {G}alerkin approximation to the {Navier-Stokes}
  equations}, Mathematics of Computation, 39 (1982), pp.~339--375.

\bibitem{besier2012}
{\sc M.~Besier and R.~Rannacher}, {\em Goal-oriented space–time adaptivity in
  the finite element {G}alerkin method for the computation of nonstationary
  incompressible flow}, International Journal for Numerical Methods in Fluids,
  70 (2012), pp.~1139--1166.

\bibitem{Calvo1990}
{\sc M.~Calvo, T.~Grande, and R.~D. Grigorieff}, {\em On the zero stability of
  the variable order variable stepsize {BDF-Formulas}}, Numerische Mathematik,
  57 (1990), pp.~39--50.

\bibitem{CHARNYI2017289}
{\sc S.~Charnyi, T.~Heister, M.~A. Olshanskii, and L.~G. Rebholz}, {\em On
  conservation laws of {Navier-Stokes Galerkin} discretizations}, Journal of
  Computational Physics, 337 (2017), pp.~289 -- 308.

\bibitem{chorin1968}
{\sc A.~J. Chorin}, {\em Numerical solution of the {Navier-Stokes} equations},
  Mathematics of Computation, 22 (1968), pp.~745--762.

\bibitem{conte}
{\sc S.~Conte and C.~de~Boor}, {\em Elementary Numerical Analysis},
  McGraw-Hill, 3rd~ed., 1980.

\bibitem{dahlquist1963}
{\sc G.~Dahlquist}, {\em A special stability problem for linear multistep
  methods}, BIT Numerical Mathematics, 3 (1963), pp.~27--43.

\bibitem{dahlquist78}
\leavevmode\vrule height 2pt depth -1.6pt width 23pt, {\em G-stability is
  equivalent to {A}-stability}, BIT Numerical Mathematics, 18 (1978),
  pp.~384--401.

\bibitem{dahlquist79}
\leavevmode\vrule height 2pt depth -1.6pt width 23pt, {\em Some properties of
  linear multistep and one-leg methods for ordinary differential equations},
  tech. rep., Royal Institute of Technology, Stockholm, 1979.

\bibitem{dahlquist81}
\leavevmode\vrule height 2pt depth -1.6pt width 23pt, {\em {On the local and
  global errors of one-leg methods}}, Tech. Rep. TRITA-NA-8110, Royal Institute
  of Technology, Stockholm, 1981.

\bibitem{dahlquist83}
\leavevmode\vrule height 2pt depth -1.6pt width 23pt, {\em On one-leg multistep
  methods}, SIAM Journal on Numerical Analysis, 20 (1983), pp.~1130--1138.

\bibitem{DECARIA2017733}
{\sc V.~DeCaria, W.~Layton, and M.~McLaughlin}, {\em A conservative, second
  order, unconditionally stable artificial compression method}, Computer
  Methods in Applied Mechanics and Engineering, 325 (2017), pp.~733 -- 747.

\bibitem{douglas}
{\sc J.~Douglas and T.~Dupont}, {\em Galerkin methods for parabolic equations},
  SIAM Journal on Numerical Analysis, 7 (1970), pp.~575--626.

\bibitem{FAILER2018448}
{\sc L.~Failer and T.~Wick}, {\em Adaptive time-step control for nonlinear
  fluid-structure interaction}, Journal of Computational Physics, 366 (2018),
  pp.~448 -- 477.

\bibitem{fiordilino2017}
{\sc J.~Fiordilino}, {\em A second order ensemble timestepping algorithm for
  natural convection}, 56 (2017), p.~816–837.

\bibitem{girault}
{\sc V.~Girault and P.-A. Raviart}, {\em Finite Element Approximation of the
  {N}avier-{S}tokes Equations}, Springer-Verlag Berlin Heidelberg, 1979.

\bibitem{griffiths}
{\sc D.~F. Griffiths and D.~J. Higham}, {\em Numerical Methods for Ordinary
  Differential Equations}, Springer-Verlag London Limited, 2010.

\bibitem{guermond2015}
{\sc J.~Guermond and P.~Minev}, {\em High-order time stepping for the
  incompressible {N}avier-{S}tokes equations}, SIAM Journal on Scientific
  Computing, 37 (2015), pp.~A2656--A2681.

\bibitem{Guglielmi2001}
{\sc N.~Guglielmi and M.~Zennaro}, {\em On the zero-stability of variable
  stepsize multistep methods: the spectral radius approach}, Numerische
  Mathematik, 88 (2001), pp.~445--458.

\bibitem{guzel}
{\sc A.~Guzel and W.~Layton}, {\em Analysis of the effect of time filters on
  the implicit method: increased accuracy and improved stability},  (2017).

\bibitem{HAY2015151}
{\sc A.~Hay, S.~Etienne, D.~Pelletier, and A.~Garon}, {\em hp-adaptive time
  integration based on the {BDF} for viscous flows}, Journal of Computational
  Physics, 291 (2015), pp.~151 -- 176.

\bibitem{hill2017}
{\sc A.~Hill and T.~J.~T. Norton}, {\em Filters for time-stepping methods},
  tech. rep., 09 2017.
\newblock International Conference on Scientific Computation and Differential
  Equations at Bath, UK.

\bibitem{ingram}
{\sc R.~Ingram}, {\em A new linearly extrapolated {Crank-Nicolson}
  time-stepping scheme for the {NSE}}, 82 (2013), pp.~1953--1973.

\bibitem{isaacson}
{\sc E.~Isaacson and H.~Keller}, {\em Analysis of Numerical Methods}, John
  Wiley and Sons, 1966.

\bibitem{janssen}
{\sc M.~Janssen and P.~Van~Hentenryck}, {\em Precisely {A}($\alpha$)-stable
  one-leg multistep methods}, BIT Numerical Mathematics, 43 (2003),
  pp.~761--774.

\bibitem{jiang2017}
{\sc N.~Jiang}, {\em A second-order ensemble method based on a blended backward
  differentiation formula timestepping scheme for time-dependent
  {N}avier-{S}tokes equations}, Numerical Methods for Partial Differential
  Equations, 33, pp.~34--61.

\bibitem{JOHN2010514}
{\sc V.~John and J.~Rang}, {\em Adaptive time step control for the
  incompressible {Navier-Stokes} equations}, Computer Methods in Applied
  Mechanics and Engineering, 199 (2010), pp.~514 -- 524.

\bibitem{kay2010}
{\sc D.~Kay, P.~M.~Gresho, D.~Griffiths, and D.~Silvester}, {\em Adaptive
  time-stepping for incompressible flow part ii: {Navier-Stokes} equations}, 32
  (2010), p.~111–128.

\bibitem{klein2015}
{\sc B.~Klein, F.~Kummer, M.~Keil, and M.~Oberlack}, {\em An extension of the
  {SIMPLE} based discontinuous {G}alerkin solver to unsteady incompressible
  flows}, International Journal for Numerical Methods in Fluids, 77,
  pp.~571--589.

\bibitem{Kulikov2006}
{\sc G.~Y. Kulikov and S.~K. Shindin}, {\em One-leg variable-coefficient
  formulas for ordinary differential equations and local--global step size
  control}, Numerical Algorithms, 43 (2006), pp.~99--121.

\bibitem{litrenchea}
{\sc Y.~Li and C.~Trenchea}, {\em A higher-order {R}obert-{A}sselin type time
  filter}, J. Comput. Phys., 259 (2014), pp.~23--32.

\bibitem{Nevanlinna1978}
{\sc O.~Nevanlinna and W.~Liniger}, {\em Contractive methods for stiff
  differential equations part i}, BIT Numerical Mathematics, 18 (1978),
  pp.~457--474.

\bibitem{NIGRO2014136}
{\sc A.~Nigro, C.~D. Bartolo, F.~Bassi, and A.~Ghidoni}, {\em Up to sixth-order
  accurate {A}-stable implicit schemes applied to the {Discontinuous Galerkin}
  discretized {Navier-Stokes} equations}, Journal of Computational Physics, 276
  (2014), pp.~136 -- 162.

\bibitem{ravindran2015}
{\sc S.~S. Ravindran}, {\em An analysis of the blended three-step backward
  differentiation formula time-stepping scheme for the {Navier-Stokes}-type
  system related to {S}oret convection}, Numerical Functional Analysis and
  Optimization, 36 (2015), pp.~658--686.

\bibitem{Shampine2005}
{\sc L.~F. Shampine}, {\em Error estimation and control for odes}, Journal of
  Scientific Computing, 25 (2005), pp.~3--16.

\bibitem{vatsa2010}
{\sc V.~N. Vatsa, M.~H. Carpenter, and D.~P. Lockard}, {\em Re-evaluation of an
  optimized second order backward difference ({BDF2OPT}) scheme for unsteady
  flow applications}, Tech. Rep. AIAA Paper 2010-0122, National Aeronautics and
  Space Administration, 2010.

\bibitem{Wang2008}
{\sc D.~Wang and S.~J. Ruuth}, {\em Variable step-size implicit-explicit linear
  multistep methods for time-dependent partial differential equations}, Journal
  of Computational Mathematics, 26 (2008), pp.~838--855.

\bibitem{wang2011}
{\sc X.~Wang}, {\em An efficient second order in time scheme for approximating
  long time statistical properties of the two dimensional {Navier-Stokes}
  equations}, 121 (2011), pp.~753--779.

\bibitem{Williams2009}
{\sc P.~D. Williams}, {\em A proposed modification to the {Robert-Asselin} time
  filter}, 137 (2009).

\bibitem{williams2011}
\leavevmode\vrule height 2pt depth -1.6pt width 23pt, {\em The {RAW} filter: An
  improvement to the {Robert-Asselin} filter in semi-implicit integrations},
  Monthly Weather Review, 139 (2011), pp.~1996--2007.

\bibitem{williams2013}
\leavevmode\vrule height 2pt depth -1.6pt width 23pt, {\em Achieving
  seventh-order amplitude accuracy in leapfrog integrations}, Monthly Weather
  Review, 141 (2013), pp.~3037--3051.

\end{thebibliography}
\bibliographystyle{siam}

\end{document}